\documentclass[10pt,reqno]{amsart}
\usepackage{amsfonts,amssymb,amsmath,amscd,euscript,array,mathrsfs}

\newcommand{\dd}{\mathsf{d}}

\newcommand{\g}{\mathfrak}
\newcommand{\ood}{{\overline{1}}}
\newcommand{\eev}{{\overline{0}}}
\newcommand{\End}{\mathrm{End}}
\newcommand{\Lie}{\mathrm{Lie}}
\newcommand{\Ad}{\mathrm{Ad}}
\newcommand{\ad}{\mathrm{ad}}
\newcommand{\sseq}{\subseteq}
\newcommand{\N}{\mathbb N}
\newcommand{\R}{\mathbb R}
\newcommand{\C}{\mathbb C}
\newcommand{\K}{\mathbb K}
\newcommand{\eps}{\varepsilon}
\newcommand{\nrm}{\mathsf q}
\newcommand{\tr}{\mathrm{tr}}
\newtheorem{theorem}{\textbf{Theorem}}[section]
\newtheorem{deff}[theorem]{\textbf{Definition}}
\newtheorem{proposition}[theorem]{\textbf{Proposition}}
\newtheorem{lemma}[theorem]{\textbf{Lemma}}
\newtheorem*{notation}{\textbf{Notation}}
\newcommand{\fctr}{\mathsf}

\newenvironment{definition}{\begin{deff}\rmfamily\upshape}{\end{deff}}
\newtheorem{corollary}[theorem]{\textbf{Corollary}}
\newtheorem{remk}[theorem]{\textbf{Example}}
\newtheorem*{remark}{\textbf{Remark}}
\newenvironment{rmk}{\begin{remark}\rmfamily\upshape }{\end{remark}}
\newenvironment{rmmk}{\begin{remk}\rmfamily\upshape}{\end{remk}}
\newenvironment{nota}{\begin{notation}\rmfamily\upshape}{\end{notation}}

\newcommand{\oline}{\overline} 
 
\newcommand{\cH}{\mathscr H}

\newcommand{\subeq}{\subseteq} 
\newcommand{\la}{\langle}
\newcommand{\ra}{\rangle}

\begin{document}

\title[categories and restriction functors]{Categories of unitary representations of Banach--Lie supergroups and restriction functors}
\author{St\'{e}phane Merigon,\ \ Karl--Hermann Neeb,\ \ Hadi Salmasian}
\address{St\'{e}phane Merigon : Department of Mathematics\\ FAU Erlangen--N\"{u}rnberg\\ 
Bismarckstrasse 1 1/2, 91054 Erlangen, Germany }
\email{merigon@mi.uni-erlangen.de}

\address{Karl--Hermann Neeb : Department of Mathematics\\ FAU Erlangen--N\"{u}rnberg\\ 
Bismarckstrasse 1 1/2, 91054 Erlangen, Germany}
\email{neeb@mi.uni-erlangen.de}

\address{Hadi Salmasian : Department of Mathematics and Statistics\\
University of Ottawa\\
585 King Edward Ave.\\
Ottawa, ON, K1N 6N5\\
Canada}
\email{hsalmasi@uottawa.ca}
\thanks{Hadi Salmasian was supported by a Discovery Grant from the Natural Sciences and Engineering Research Council of Canada and an Alexander von Humboldt Fellowship for Experienced Researchers.}

\maketitle

\begin{abstract}
We prove that the categories of smooth and analytic unitary 
representations of Banach--Lie supergroups are well-behaved under restriction functors, in the sense that the
restriction of a representation to an integral subsupergroup is well-defined. We also prove that the category of analytic representations is isomorphic to a subcategory of the category of smooth representations. 
These facts  are needed as a 
crucial first step to a rigorous treatment of the analytic theory of unitary representations
 of Banach--Lie supergroups. They extend the known results for finite dimensional Lie supergroups. In the infinite dimensional case the proofs require several new ideas. As an application, we give an analytic realization of the oscillator representation of the restricted orthosymplectic Banach--Lie supergroup.

\end{abstract}
\section{Introduction}
In the last two decades, unitary representations of finite and infinite dimensional Lie supergroups and Lie superalgebras have received  growing interest from both mathematicians and physicists.
These unitary representations appear in  
the classification of free relativistic superparticles in SUSY quantum mechanics 
(for example, see \cite{FSZ81} and \cite{SS74}) 
which relies on the \emph{little supergroup method}, an idea originating from the classical works 
of Mackey and Wigner.  
Unitary representations of the $N=1$ super Virasoro algebras were classified by  Friedan--Qiu--Shenker, Goddard--Kent--Olive, and Sauvageot
\cite{friedanqiu}, \cite{goddard}, \cite{sauvageot}.
For the $N=2$ super Virasoro algebras, the results are due to Boucher--Friedan--Kent and Iohara
\cite{friedankent2}, \cite{iohara}.
 Kac and Todorov classified unitary representations of 
superconformal current algebras \cite{kactodor}. 
Using an analogue of the Sugawara construction, Jarvis and Zhang 
constructed unitary representations of untwisted affine Lie superalgebras \cite{jarviszhang}. 
Unitary highest weight modules of affine Lie superalgebras were also studied by Jakobsen and Kac  \cite{jakobsen2}.

Much of the research done on unitary representations 
is algebraic, i.e., studies them as unitarizable modules over Lie superalgebras. A mathematically rigorous investigation of 
analytic aspects of unitary representations is more recent
\cite{allhilglaub}, \cite{varadarajan}, \cite{salmasian}. 

In \cite{varadarajan} the authors propose an approach to harmonic analysis on Lie supergroups.
One key idea in their work is to use the equivalence between
the category of Lie supergroups
and the category of \emph{Harish--Chandra pairs} 
$(G,\g g)$
\cite[Sec. 3.8]{delignemorgan}, \cite[Sec. 3.2]{kostant}.
The advantage of using this equivalence is that for Harish--Chandra pairs 
the definition of a unitary representation is more concrete. 
Roughly speaking, 
a Harish--Chandra pair is an ordered pair $(G,\g g)$ where $G$ is a Lie group, 
$\g g=\g g_\eev\oplus\g g_\ood$ is a Lie superalgebra, $\g g_\eev=\Lie(G)$, and 
there is an action of $G$ on $\g g$ which is compatible with the adjoint action of $\g g$.
(For a precise definition, see Definition \ref{def-blsupergroup}.)
A unitary representation of $(G,\g g)$ is a triple $(\pi,\mathscr H,\rho^\pi)$ where $(\pi,\mathscr H)$ 
is a unitary representation of the Lie group $G$ (in the sense of \cite[Sec. 1.2]{varabook}) and $\rho^\pi$ is a  representation of
the Lie superalgebra $\g g$, realized on a dense subspace of $\mathscr H$ 
(consisting of smooth vectors) which is compatible with $\pi$ on $\g g_\eev$ (see Definition 
\ref{defi-smoothanalytic}).

An important observation
in \cite{varadarajan} is that if $x\in\g g_\ood$ then 
\[
\rho^\pi(x)^2=\frac{1}{2}\rho^\pi([x,x])=
\frac{1}{2}\dd\pi([x,x])
\] which suggests that
$\rho^\pi(x)$ should be
an unbounded operator on $\mathscr H$, i.e., it can only be densely defined. 
Therefore one needs to fix a common domain for the operators $\rho^\pi(x)$.
For instance, one can choose the common domain 
to be $\mathscr H^\infty$ or $\mathscr H^\omega$, the subspaces of smooth or analytic vectors of
the unitary representation
$(\pi,\mathscr H)$, which lead to categories $\fctr{Rep}^\infty(G,\g g)$ and $\fctr{Rep}^\omega(G,\g g)$ of smooth and analytic representations of 
$(G,\g g)$. With any choice of such a common domain, we are lead to the following two questions.
\begin{enumerate}
\item[(i)] What is the relation between the categories of unitary representations of $(G,\g g)$ when the common domain for the realization of $\rho^\pi$ varies? For instance, what is the relation between 
$\fctr{Rep}^\infty(G,\g g)$ and $\fctr{Rep}^\omega(G,\g g)$? 

Here the issue is that if $(\pi,\rho^\pi,\mathscr H)$ is an object of $\fctr{Rep}^\omega(G,\g g)$, then 
the action of $\g g_\ood$ is defined on $\mathscr H^\omega$. In general $\mathscr H^\omega\subsetneq\mathscr H^\infty$ and therefore a~priori it is not 
obvious why $(\pi,\rho^\pi,\mathscr H)$ 
is also an object of $\fctr{Rep}^\infty(G,\g g)$.

\item[(ii)] Let $(H,\g h)$ be a subsupergroup \footnote{Because of the equivalence between 
the categories of Lie supergroups and Harish--Chandra pairs, it will be harmless 
(and simpler for our presentation) 
to refer to a Harish--Chandra pair as a Lie supergroup. When the Harish--Chandra pair is modeled on a Banach space, we will call it a Banach--Lie supergroup (see Definition \ref{def-blsupergroup}).} of $(G,\g g)$. Are there well-defined restriction functors 
\[
\fctr{Res}^\infty:\fctr{Rep}^\infty(G,\g g)\mapsto \fctr{Rep}^\infty(H,\g h)\]
 and 
\[
\fctr{Res}^\omega:\fctr{Rep}^\omega(G,\g g)\mapsto\fctr{Rep}^\omega(H,\g h)?
\] 

Here the issue 
is the following. Let $(\pi,\rho^\pi,\mathscr H)$ be a unitary representation of $(G,\g g)$. 
Denote the subspace of smooth vectors of $(\pi,\mathscr H)$ (respectively, $(\pi\big|_H,\mathscr H)$) by
$\mathscr H^\infty_G$ (respectively, $\mathscr H^\infty_H$). In general
$\mathscr H^\infty_G\subsetneq\mathscr H^\infty_H$, and 
a priori the operators $\rho^\pi(x)$, where $x\in\g h_\ood$, are only defined on $\mathscr H^\infty_G$. Consequently, 
restricting the actions naively does not lead to an object of the category $\fctr{Rep}^\infty(H,\g h)$.
\end{enumerate}
The answers to the above questions are crucial to obtaining well-behaved categories of unitary representations for
Harish--Chandra pairs. For (finite dimensional) Lie supergroups they
 are addressed in \cite{varadarajan}.

The main goal of this article is to answer the above questions for Harish--Chandra pairs associated to Banach--Lie supergroups. Molotkov's work (which is extended in Sachse's Ph.D. thesis) develops a functorial 
theory of Banach--Lie supermanifolds which specializes to the  (finite dimensional) 
Berezin--Leites--Kostant theory  
and the (possibly infinite dimensional) DeWitt--Tuynman theory 
\cite{alldridge}. In the functorial approach one can 
associate Harish--Chandra pairs to Banach--Lie supergroups as well. 
We believe that the success of the Harish--Chandra pair approach 
in the finite dimensional case justifies
their use in studying harmonic analysis on Banach--Lie supergroups.

The direction we choose for the formulation and proofs of our results is 
similar in spirit to \cite{varadarajan}. However, some of 
the arguments used in \cite{varadarajan}, 
in particular the proofs of \cite[Prop. 1 and 2]{varadarajan},
depend crucially on finiteness of dimension. In this article we present new arguments which generalize to the Banach--Lie case. One of our tools is 
the theory of analytic maps between Banach spaces. 
Of the multitude of existing variations of this theory, 
the most relevant to this article is the work of 
Bochnak and Siciak \cite{bochnaksiciak1}, \cite{bochnaksiciak}, \cite{siciakLe}. Their work generalizes
the results in the book by Hille and Phillips \cite{hillephilips}.

We also present two applications of our techniques. The first application
is that when $g$ is in the connected component of identity of $G$, the conjugacy invariance relation 
\begin{equation}
\label{conjugeqn}
\pi(g)\rho^\pi(x)\pi(g)^{-1}=\rho^\pi(\Ad(g)x) \text{ for every }x\in\g g
\end{equation}
follows from the remaining assumptions 
in the definition 
of a unitary representation of $(G,\g g)$.
(See Proposition \ref{proposition-conjugacy} below.)
Such a conjugacy invariance relation  
is one of the assumptions of \cite[Def. 2]{varadarajan}. 
As a consequence, we obtain a
reformulation of the definition of a unitary representation (see Definition \ref{defi-smoothanalytic}) which makes
it more practical than the original form given in \cite[Def. 2]{varadarajan}, because in explicit examples checking that the infinitesimal action satisfies the bracket relation is easier than checking the conjugacy invariance relation \eqref{conjugeqn}.

The second application is an analytic realization of the oscillator representation of the restricted orthosymplectic Banach--Lie supergroup. Again the main 
issue is to show that the action of the odd part is defined on the subspaces of 
smooth and analytic vectors and leaves them invariant. We use a general statement, i.e., Theorem \ref{thm-stabili}, which we expect to be useful in a variety of situations, for instance when one is interested in integrating a representation of a Banach--Lie superalgebra.

This article is organized as follows. In Section \ref{sec-normoncomplex} we introduce our notation and basic definitions and prove some general lemmas which will be used in the later sections. In Section 
\ref{smoothvec} we state some general facts about smooth and analytic vectors of unitary representations of Banach--Lie groups. Section \ref{sec-blsupergps} is devoted to the proof of our main results, Theorems 
\ref{thm-stabili} and  \ref{thm-restfunct}. 
In Section \ref{sec-restricted} we give a realization of the oscillator representation of the restricted orthosymplectic 
Banach--Lie supergroup.
In Appendix \ref{exampleanalytic} we give an example of a smooth unitary representation of a Banach--Lie group which has no nonzero analytic vectors. In Appendix \ref{examplebounded} we give an example of an analytic 
unitary representation of a Banach--Lie group which has no nonzero bounded vectors.
Appendix \ref{section-appendix-analyticmaps} contains the background material on analytic maps between Banach spaces.

This article contains general results about arbitrary Banach--Lie supergroups. These results 
are needed for the study of unitary representations of 
concrete examples. 
In our forthcoming works we will study 
unitary representations of Banach--Lie supergroups corresponding to affine Lie superalgebras, and
the supergroup version of the Kirillov--Ol'shanskii classification of unitary representations of 
the infinite dimensional unitary group \cite{kirillov}, \cite{olshanski}.

\section{Notation and preliminaries}
\label{sec-normoncomplex}
If $\mathscr B$ is a real Banach space with norm $\|\cdot\|_\R$
then we define its complexification as the complex Banach space with underlying space 
$\mathscr B^\C=\mathscr B\otimes_\R\C$ 
and with norm $\|\cdot\|_\C$, where for every $v\in\mathscr B^\C$ we set 
\begin{equation*}
\|v\|_\C=\inf\{\ |\zeta_1|\cdot\|v_1\|_\R+\cdots+|\zeta_k|\cdot\|v_k\|_\R\ : \ v=v_1\otimes_\R\zeta_1+\cdots+v_k\otimes_\R\zeta_k\ \}.
\end{equation*}
Note that if $v=v_1\otimes_\R 1+v_2\otimes_\R i$ then 
\begin{equation}
\label{equivalenceofnorms}
\max\{\|v_1\|_\R,\|v_2\|_\R\}\leq \|v\|_\C\leq \|v_1\|_\R+\|v_2\|_\R.
\end{equation}

All Banach and Hilbert spaces will be separable. Let $\mathscr H$ be a real or complex Hilbert space. 
If $T:\mathscr H\to \mathscr H$ is a bounded linear operator on $\mathscr H$, then  $\|T\|_\mathrm{Op}$
denotes the operator norm of $T$, and if $T$ is a Hilbert--Schmidt operator, then $\|T\|_\mathrm{HS}$ denotes its Hilbert--Schmidt norm. 

Now assume $\mathscr H$ is a complex Hilbert space. The group of unitary linear transformations on $\mathscr H$ is denoted by $\mathrm U(\mathscr H)$.
The domain of an unbounded linear operator $T$ on $\mathscr H$ is denoted by $\mathcal D(T)$,
and if $\mathscr B\sseq \mathcal D(T)$ is a subspace, then $T\big|_\mathscr B$ denotes 
the restriction of $T$ to $\mathscr B$. If $S$ and $T$ are two unbounded operators on $\mathscr H$ then
their sum $S+T$ is an operator with domain $\mathcal D(S+T)=\mathcal D(S)\cap\mathcal D(T)$, and their product
$ST$ is an operator with domain $\mathcal D(ST)=\{\,v\in\mathcal D(T)\,:\,Tv\in\mathcal D(S)\,\}$.
For two unbounded operators $S$ and $T$, we write 
$S\prec T$ if $\mathcal D(S)\sseq D(T)$ and $T\big|_{\mathcal D(S)}=S$.

The adjoint of a linear operator $T$ is denoted by $T^*$. If $T$ is closable, then
its closure is denoted by $\overline T$.
For every integer $n>1$ we set \[
\mathcal D(T^n)=\{\,v\in\mathcal D(T)\ :\ Tv\in\mathcal D(T^{n-1})\,\}.
\] 
We also set $\mathcal D^\infty(T)=\bigcap_{n=1}^\infty D(T^n)$. If $v\in \mathcal D^\infty(T)$ satisfies
\[
\sum_{n=0}^\infty \frac{t^n}{n!}\|T^nv\|<\infty \text{ for some $t>0$}
\]
then $v$ is called an \emph{analytic vector} of $T$. The space of analytic vectors of $T$ is denoted 
by $\mathcal D^\omega(T)$.

By a $\mathbb{Z}_2$-graded Hilbert space $\mathscr H=\mathscr H_\eev\oplus\mathscr H_\ood$ we simply mean 
an orthogonal direct sum of two complex Hilbert spaces
$\mathscr H_\eev$ and $\mathscr H_\ood$. (For another equivalent definition, see \cite[Sec. 2.1]{varadarajan}.)

All Banach--Lie groups are real analytic.
Let $G$ be a Banach--Lie group and $\g g=\Lie(G)$.
The identity component of $G$ is denoted by $G^\circ$.
We assume, without loss of generality, that the norm inducing the topology of the Banach--Lie algebra $\g g$ satisfies  
$\|[x,y]\|\leq \|x\|\cdot\|y\|$ for every $x,y\in\g g$.

By a \emph{Banach--Lie superalgebra} we mean a Lie superalgebra $\g g=\g g_\eev\oplus \g g_\ood$ 
over $\R$ or $\C$ with the following two properties.
\begin{itemize}
\item[(i)] $\g g$ is a Banach space with a norm $\|\cdot\|$ satisfying 
\begin{equation}
\label{eqn-normineq}
\|[x,y]\|\leq \|x\|\cdot\|y\|\text{ for every }x,y\in\g g.
\end{equation}
\item[(ii)] $\g g_\eev$ and $\g g_\ood$ are closed subspaces of $\g g$. 
\end{itemize}
\begin{rmk}
Consider the norm $\|\cdot\|'$ on $\g g$ which is defined as follows. For every $x\in \g g$, we write
$x=x_\eev+x_\ood$ where $x_\eev\in \g g_\eev$ and $x_\ood\in \g g_\ood$, and set 
$\|x\|'=\|x_\eev\|+\|x_\ood\|$.
From the definition of a Banach--Lie superalgebra it follows that the norms $\|\cdot\|$ and $\|\cdot\|'$ are equivalent.
\end{rmk}
The Banach--Lie group of continuous even  automorphisms (i.e., automophisms which preserve parity) 
of 
a Banach--Lie superalgebra $\g g$ is denoted by
$\mathrm{Aut}(\g g)$.  
\begin{definition}
\label{def-blsupergroup}
A \emph{Banach--Lie supergroup} is an ordered pair $(G,\g g)$ with the following properties.
\begin{itemize}
\item[(i)] $G$ is a Banach--Lie group.
\item[(ii)] $\g g$ is a Banach--Lie superalgebra over $\R$.
\item[(iii)] $\g g_\eev=\mathrm{Lie}(G)$. 
\item[(iv)] There exists a morphism of Banach--Lie groups 
$\Ad:G\to\mathrm{Aut}(\g g)$ 
such that 
\[
\dd_e\Ad(x)=\ad_x \text{ for every }x\in\g g_\eev,
\]
where $\dd_e\Ad$ denotes the differential of $\Ad$ at $e\in G$, and $\ad_x(y)=[x,y]$.
\end{itemize}
\end{definition}
We refer to the morphism $\Ad:G\to \mathrm{Aut}(\g g)$ of Definition \ref{def-blsupergroup}(iv) by the \emph{adjoint action} of 
$G$ on $\g g$.
Observe that the map 
\[
G\times \g g\to \g g\; ,\; (g,x)\mapsto \Ad(g)x
\]
is analytic. 

If $\g g$ is a
real Lie superalgebra, then its complexification is denoted by $\g g^\C$. 
It is easily seen that if $\g g$ is a real Banach--Lie superalgebra, then 
$\g g^\C$ with the norm defined in the beginning of Section \ref{sec-normoncomplex}  is a complex Banach--Lie superalgebra. After a suitable scaling, we can
assume that the norm $\|\cdot\|$ chosen on $\g g^\C$ satisfies 
\[
\|[x,y]\|\leq\|x\|\cdot\|y\| \text{ for every }x,y\in\g g^\C.
\] 

Let $(G,\g g)$ be a Banach--Lie supergroup. An \emph{integral subsupergroup} of $(G,\g g)$ is a Lie supergroup
$(H,\g h)$ with a morphism $(\phi,\varphi):(H,\g h)\to(G,\g g)$ such that $\phi:H\to G$ is an injective homomorphism of
Banach--Lie groups,
$\varphi:\g h\to\g g$ is a continuous injective homomorphism of Lie superalgebras, and 
$\dd_e\phi=\varphi\big|_{\g h_\eev}$.

A \emph{unitary representation} of a Banach--Lie group $G$ is an ordered pair $(\pi,\mathscr H)$ such that 
$\pi:G\to\mathrm U(\mathscr H)$ is a group homomorphism and
for every $v\in\mathscr H$ the orbit map 
\begin{equation}
\label{eqn:orbitmap}
\pi^v:G\to\mathscr H\;,\; \pi^v(g)=\pi(g)v
\end{equation}
is continuous.
The restriction of $(\pi,\mathscr H)$ to a subgroup $H$ of $G$ is denoted by $(\pi\big|_H,\mathscr H)$.

We conclude this section with a few lemmas about unbounded operators.

\begin{lemma}
\label{lem-criteriaforesa}
Let $T$ be a self-adjoint operator on a complex Hilbert space $\mathscr H$ and 
$\mathscr L\sseq \mathcal D(T)$
be a dense subspace of $\mathscr H$ satisfying at least one of the following properties.
\begin{enumerate}
\item[(a)] For every $t\in\R$, we have $e^{itT}\mathscr L\sseq \mathscr L$.
\item[(b)] Every $v\in \mathscr L$ is an analytic vector for $T$.
\end{enumerate} 
Then $T\big|_\mathscr L$ is essentially self-adjoint.
\end{lemma}
\begin{proof}
When (a) holds, the result follows from \cite[Thm. VIII.11]{reedsimon}. When (b) holds, it is an immediate consequence of Nelson's Analytic Vector Theorem \cite[Lem. 5.1]{nelson}.\qedhere
\end{proof}
The next lemma is obvious, but it will help us shorten several similar arguments. 
\begin{lemma}
\label{lem-helpshorten}
Let $P_1$ and $P_2$ be symmetric linear operators on a complex Hilbert space $\mathscr H$ and 
$v\in\mathcal D(P_1)\cap\mathcal D(P_2)$ such that $P_1v\in\mathcal D(P_1)$ and $P_1^2v=P_2v$. 
Then \[
\|P_1v\|\leq \|v\|^\frac{1}{2}\cdot\|P_2v\|^\frac{1}{2}.
\]
\end{lemma}
\begin{proof}
$\|P_1v\|^2=|\langle P_1v,P_1v\rangle|=|\langle v,P_1^2v\rangle|=|\langle v,P_2v\rangle|\leq\|v\|
\cdot \|P_2v\|$.
\qedhere
\end{proof}

\begin{lemma}
\label{varalemma1}

Let $T$ be a self-adjoint operator on a complex Hilbert space $\mathscr H$. Let
$\mathscr L$ be a dense subspace of $\mathscr H$
such that $\mathscr L\sseq\mathcal D(T)$ and  $T\big|_{\mathscr L}$ is essentially self-adjoint, 
and $S$ be a symmetric operator such that $\mathscr L\sseq\mathcal D(S)$, $S\mathscr L\sseq \mathscr L$,  and 
$S^2\big|_{\mathscr L}=T\big|_{\mathscr L}$. Then $S\big|_\mathscr L$ is essentially self-adjoint, 
$\overline{S\big|_\mathscr L}=\overline S$, and $\overline{S}^2=T$.


\end{lemma}

\begin{proof}
The proof of the lemma is similar to the proof of \cite[Lem. 1]{varadarajan}. However, for the sake of
completeness we give a complete proof. 

Set $S_1=S\big|_\mathscr L$.  
From $T=\overline{T\big|_\mathscr L}=\overline{S_1^2}$ it follows that 
\begin{equation}
\label{eqn-positv}
\langle Tv,v\rangle \geq 0
\text{ for every $v\in\mathcal D(T)$.}
\end{equation}
By \cite[Thm. VIII.3]{reedsimon}, in order
to prove that $S_1$ is essentially self-adjoint, it suffices to show
that if 
\begin{equation}
\label{eqn-vlambdaeqn}
S_1^*v=\lambda v
\end{equation}
for a nonzero
$v\in\mathcal D(S_1^*)$ and a $\lambda\in\C$, then $\lambda\in\R$.
If $v$ and $\lambda$ satisfy \eqref{eqn-vlambdaeqn}
then for every $w\in\mathscr L$ we have 
\begin{align*}
\langle S_1^2w,v\rangle=\langle S_1w,S_1^*v\rangle
=\overline\lambda\langle S_1w,v\rangle
=\overline\lambda\langle w,S_1^* v \rangle =\overline\lambda^2\langle w,v\rangle=\langle w,\lambda^2 v\rangle.
\end{align*}
Therefore
$v\in\mathcal D\big((S_1^2)^*\big)$ and $(S_1^2)^*v=\lambda^2v$. But 
$T=(T\big|_\mathscr L)^*=(S_1^2)^*$ and in particular $Tv=(S_1^2)^*v=\lambda^2v$. 
From \eqref{eqn-positv} it
follows immediately that $\lambda\in\R$. This completes the proof of essential self-adjointness of $S_1$.

Next observe that $\overline{S_1}\prec \overline S\prec S^*\prec S_1^*$. Since $S_1$ is essentially self-adjoint, we have $\overline{S_1}=S_1^*$ and therefore $\overline S=\overline{S_1}$. Since the operator $\overline{S_1}$ 
is self-adjoint, it follows from 
\cite[Cor. XII.2.8]{dunford} that
$\overline{S_1}^2$ is also self-adjoint. Consequently,
\[
\overline{S_1}^2=\big(\overline{S_1}^2\big)^*\prec (S_1^2)^*=T\ \text{ and }\ T=\overline{S_1^2}\prec \overline{S_1}^2
\]
which implies that $\overline S^2=\overline{S_1}^2=T$.
\qedhere
\end{proof}
\begin{rmk}
Note that in the statement of Lemma \ref{varalemma1}, it follows directly from $\overline S^2=T$ that 
$\mathcal D(T)\sseq\mathcal D(\overline S)$.
\end{rmk}

\begin{lemma}
\label{lem-two-oper}
Let $P_1$ and $P_2$ be two symmetric operators on a complex Hilbert space $\mathscr H$ such that 
$\mathcal D(P_1)=\mathcal D(P_2)$. Let $\mathscr L\sseq\mathcal D(P_1)$ 
be a dense linear subspace of $\mathscr H$ such that 
$P_1\big|_{\mathscr L}=P_2\big|_{\mathscr L}$. Assume that the latter operator is essentially self-adjoint.
Then $P_1=P_2$.
\end{lemma}
\begin{proof}
Observe that 
\[
\overline{P_1\big|_{\mathscr L}}
\prec
\overline {P_1}
\prec
P_1^*
\prec
(P_1\big|_{\mathscr L})^*
=
\overline{P_1\big|_{\mathscr L}}
\]
from which it follows that $\overline{P_1}=\overline{P_1\big|_\mathscr L}$. Similarly, $\overline{P_2}=\overline{P_2\big|_\mathscr L}$ and
therefore $\overline{P_1}=\overline{P_2}$. It follows immediately that $P_1=P_2$.
\qedhere
\end{proof}

\section{Smooth and analytic vectors of unitary representations}
\label{smoothvec}

Let $G$ be a Banach--Lie group, $\g g=\Lie(G)$, and 
$(\pi,\mathscr H)$ be a  unitary representation of $G$. 
For every $x\in\g g$, the skew-adjoint operator corresponding to the one-parameter unitary representation 
\[
\R\to\mathrm U(\mathscr H)\ ,\ t\mapsto \pi(\exp(tx))
\] 
via Stone's Theorem is denoted by $\dd\pi(x)$. If $x=a+ib\in\g g^\C$, then we set \[
\dd\pi(x)=\dd\pi(a)+i\dd\pi(b)\] where the right hand side
means the sum of two unbounded operators, i.e., 
\[
\mathcal D(\dd\pi(x))=\mathcal D(\dd(\pi(a))\cap\mathcal D(\dd\pi(b)).
\]

Recall that $\pi^v:G\mapsto \mathscr H$ denotes the orbit map defined in \eqref{eqn:orbitmap}. 
Let $\mathscr H^\infty$ be the subspace of \emph{smooth vectors} of $(\pi, \mathscr H)$, i.e.,
\[
\mathscr H^\infty=\{\ v\in\mathscr H\ : \ \pi^v\text{ is a smooth map}\ \}.
\]
If $\mathscr H^\infty$ is a dense subspace of $\mathscr H$ then the representation 
$(\pi,\mathscr H)$ is called a \emph{smooth unitary representation}.

As in \cite[Sec. 4]{khdifferentiable}, we endow the space
$\mathscr H^\infty$ with the topology induced by the family of seminorms $\{\nrm_n\}_{n=0}^\infty$ where 
\[
\nrm_n(v)=\sup\big\{\ \|\dd\pi(x_1)\cdots \dd\pi(x_n)v\|\ :\ x_1\ ,\ldots,x_n\in\g g\ ,\ \|x_1\|\leq 1,\ldots,\|x_n\|\leq 1\ \big\}.
\]
With this topology $\mathscr H^\infty$ is a Fr\'{e}chet space \cite[Prop. 5.4]{khdifferentiable}. Moreover,
the map
\begin{equation}
\label{ghinfhinf}
\g g\times\mathscr H^\infty\to\mathscr H^\infty\ \,,\,\ (x,v)\mapsto \dd\pi(x)v
\end{equation}
is continuous \cite[Lem. 4.2]{khdifferentiable}
and the map
\begin{equation}
\label{gghinfhinf}
G\times \mathscr H^\infty\to\mathscr H^\infty\ \, ,\,\  (g,v)\mapsto\pi(g)v
\end{equation}
is smooth \cite[Thm. 4.4]{khdifferentiable}.
A vector $v\in\mathscr H^\infty$ is called 
\emph{analytic} if  the orbit map $\pi^v:G\to\mathscr H$ is a real analytic function. The space of analytic vectors is denoted by  $\mathscr H^\omega$. If $\mathscr H^\omega$ is a dense subspace of $\mathscr H$ then the representation 
$(\pi,\mathscr H)$ is called an \emph{analytic unitary representation}.

Proposition \ref{thm-oneparam} below records well known facts about unitary representations of the real line and its proof is omitted. 
\begin{proposition}
\label{thm-oneparam}
Let $(\pi,\mathscr H)$ be a unitary representation of $\R$ and $A$ be the skew-adjoint 
infinitesimal generator of $(\pi,\mathcal H)$.
\begin{enumerate}
\item[(i)] A vector $v\in\mathscr H$ is smooth if and only if $v\in\mathcal D^\infty(A)$.
\item[(ii)] A vector $v\in \mathscr H$ is analytic if and only if 
$v\in\mathcal D^\omega(A)$. 
\item[(iii)] Let $v\in\mathcal D^\infty(A)$ and $r>0$ be such that 
$\displaystyle\sum_{n=0}^\infty \frac{r^n}{n!}\|A^nv\|<\infty$.
Then \[
 \pi(t)v=\sum_{n=0}^\infty \frac{t^n}{n!}A^nv\ \text{for every $t\in\,(-r,r)\,$.}
 \] 
\end{enumerate}

\end{proposition}

If $G$ is a Banach--Lie group and $\g g=\Lie(G)$ then for every $r>0$ we set 
\[
B_r=\{\ x\in\g g^\C\ :\ \|x\|<r\ \}.
\]

\begin{lemma}
\label{convoncplx}
Let $G$ be a Banach--Lie group, $\g g=\Lie(G)$, $(\pi,\mathscr H)$ be a unitary representation of $G$, and $v\in\mathscr H^\infty$.
Then $v\in\mathscr H^\omega$  if and only if there exists an $r>0$ such that for every $x\in B_r$ the series
\begin{equation}
\label{fvpower}
f_v(x)=\sum_{n=0}^\infty \frac{1}{n!}\dd\pi(x)^nv
\end{equation}
converges in $\mathscr H$ (and therefore defines an analytic map $B_r\to\mathscr H$).
\end{lemma}
\begin{proof}
Let $v\in\mathscr H^\omega$ and $x\in B_r$. By Proposition \ref{thm-oneparam}, for
all sufficiently small $t>0$ we have 
\[
\pi(\exp(tx))v=\sum_{n=0}^\infty \frac{t^n}{n!}\dd\pi(x)^nv.
\] This means that the series
\eqref{fvpower} converges in an absorbing set. By \cite[Lem. 4.4]{neebanalytic} for every integer $n\geq 0$
the function
\[
\g g\to \mathscr H\quad,\quad x\mapsto\frac{1}{n!}\dd\pi(x)^nv
\]
is a continuous homogeneous polynomial of degree $n$. Therefore by Theorem \ref{firstthmbosi}(i)
the series \eqref{fvpower} defines an analytic map in a neighborhood of
zero in $\g g$. By Theorem \ref{secondthmbosi} this series also defines an analytic map 
in a neighborhood of zero in $\g g_\C$.

Conversely, assume that the series \eqref{fvpower} converges for every $x\in B_r$.
By Theorem \ref{firstthmbosi}(i) there exists an $r'>0$ such that 
\begin{equation}
\label{unifboundfv}
\sum_{n=0}^\infty \frac{1}{n!}\sup\{\|\dd\pi(x)^nv\|\ :\ x\in B_{r'}\}<\infty.
\end{equation}
Therefore Proposition \ref{thm-oneparam}(iii) implies that
\[
\pi^v(\exp(x))=f_v(x) \text{ for every }x\in B_{r'}\cap \g g.
\] 
From Theorem \ref{firstthmbosi}(ii) it follows that $f_v$ is an analytic function in $B_{r}$. Therefore the orbit map $\pi^v$ is also analytic in a neighborhood of identity of $G$. It follows immediately that $v\in\mathscr H^\omega$.
\qedhere

\end{proof} 
\begin{nota}
For every $r>0$ set
\[
\mathscr H^{\omega,r}=\{\ v\in \mathscr H^\infty\ :\ \sum_{n=0}^\infty \frac{1}{n!}\dd\pi(x)^n v\text{ converges in $\mathscr H$ for every $x\in B_r$}
\ \}.
\]
\end{nota}
From Lemma \ref{convoncplx} it follows that $\mathscr H^\omega=\bigcup_{r>0}\mathscr H^{\omega,r}$, and 
Proposition \ref{thm-oneparam}(iii) shows that if 
$v\in\mathscr H^{\omega,r}$ then 
\begin{equation}
\label{eqn-piexpxv}
\pi(\exp(x))v=f_v(x)\,\text{ for every }\,x\in B_r\cap \g g.
\end{equation}
\begin{lemma}
\label{analyticBr}
Let $G$ be a Banach--Lie group,  $(\pi,\mathscr H)$ be a unitary representation of $G$,
$v\in\mathscr H^\infty$, $r>0$, and $\g g=\Lie(G)$. Then $v\in\mathscr H^{\omega,r}$ if and only if the map 
$\pi^v\circ\exp\big|_{B_r\cap \g g}$ extends to an analytic function $h_v:B_r\to \mathscr H$.
\end{lemma}

\begin{proof}
Let $v\in \mathscr H^{\omega,r}$. 
From Theorem 
\ref{firstthmbosi}(ii) 
it follows that the series $f_v(x)$ of \eqref{fvpower}  defines an analytic 
function in $B_r$.
By \eqref{eqn-piexpxv}
we have 
\[
\pi(\exp(x))v=f_v(x) \text{ for every }x\in B_r\cap \g g.
\]
Therefore $\pi^v\circ\exp\big|_{B_r\cap \g g}$ extends to an analytic function in $B_r$.

Conversely, assume that $\pi^v\circ\exp\big|_{B_r\cap \g g}$ extends to an analytic map $h_v:B_r\to \mathscr H$.
Since $B_r$ is a balanced neighbourhood of zero, from Theorem \ref{fourththmbosi} it follows that 
\[
h_v(x)=\sum_{n=0}^\infty \frac{1}{n!}\delta_0^{(n)}h_v(x)\ \text{ for every $x\in B_r$}.
\]
By Theorem \ref{fourththmbosi}(i),  for every $n\geq 0$ the function
$\delta_0^{(n)}h_v:\g g^\C\to\mathscr H$ 
is a continuous homogeneous polynomial of degree $n$.
Observe that $h_v(x)=\pi(\exp(x))v$ for every $x\in B_r\cap \g g$, and by taking
the $n$-th directional derivatives of both sides  we obtain
\begin{equation}
\label{eqn-nthderiv}
\delta_0^{(n)}h_v(x)=\dd\pi(x)^nv\ \text{ for every $x\in B_r\cap \g g$.}
\end{equation}
Both sides of \eqref{eqn-nthderiv} are continuous homogeneous polynomials, and in particular analytic in $B_r$. 
Therefore by analytic continuation,
the equality \eqref{eqn-nthderiv} holds for every  
$x\in B_r$. 
Consequently, the series \eqref{fvpower} converges for every $x\in B_r$, i.e., $v\in\mathscr H^{\omega,r}$.
\qedhere
\end{proof}

\begin{lemma}
\label{lem-analyfv}
Let $G$ be a Banach--Lie group and $\g g=\Lie(G)$. 
Then there exists an $r_\circ>0$ such that for every $0<r<r_\circ$,  every 
unitary representation $(\pi,\mathscr H)$ of $G$, and every $v\in\mathscr H^{\omega,r}$, the following
statements hold.
\begin{itemize}
\item[(i)] $f_v(x)\in\mathscr H^{\omega}$ for every $x\in B_r$.
\item[(ii)] If $a\in\g g^\C$ then the map
\[
u_a:B_r\to \mathcal H\ ,\
u_a(x)=\dd\pi(a)(f_v(x))
\]
is analytic in $B_r$.
\end{itemize}
\end{lemma}
\begin{proof}
(i) Let $z\star z'$ denote the  Baker--Campbell--Hausdorff series for two elements $z,z'\in\g g$ whenever it converges.
Choose $r_\circ>0$ small enough such that the map 
\[
\mu:B_{r_\circ}\times B_{r_\circ}\to\g g^\C\ ,\ \mu(z,z')=z\star z'
\]
is analytic in $B_{r_\circ}\times B_{r_\circ}$.

Let $x\in B_r$. We write $x=x'+ix''$ where 
$x',x''\in\g g$, and consider the complex subspace $\mathscr V=\mathrm{Span}_\C\{x',x''\}$ of $\g g^\C$. 
Choose $s>0$ such that $\|x\|<s<r$, and let $\overline{B_s}$ denote the closure of $B_s$. Since $\mathscr V$ is finite dimensional, $\overline{B_s}\cap \mathscr V$ is a compact subset of $W$. 
It follows that 
there exists an $0<r'<r$ such that  
\[
\{\ z\star z'\ :\ z\in B_{r'}\text{ and }z'\in B_s\cap \mathscr V\ \}\sseq B_r.
\]
Consequently, for every $z\in B_{r'}$ the map
\[
\phi_z:B_s\cap\mathscr V\to \mathscr H\ ,\ 
\phi_z(y)=f_v(z\star y)
\]
is well-defined and analytic. 

Next fix $z\in B_{r'}\cap \g g$ and consider the function 
\[
\psi_z:B_s\to\mathscr H\ ,\
\psi_z(y)=\pi(\exp(z))f_v(y)\]
which is analytic in $B_s$. If $y\in B_s\cap\mathscr V\cap\g g$, then 
$z\star y\in B_r\cap \g g$ and  by
\eqref{eqn-piexpxv} we have
\[
\phi_z(y)=f_v(z\star y)=\pi(\exp(z\star y))v=\pi(\exp(z))\pi(\exp(y))v=\psi_z(y).
\]
As both $\phi_z$ and $\psi_z$ are analytic in $B_s\cap\mathscr V$, it follows that
the equality $\phi_z(y)=\psi_z(y)$ holds for every $y\in B_s\cap\mathscr V$. In particular, for every $z\in B_{r'}\cap \g g$ we have 
\begin{equation}
\label{eqnorbt}
f_v(z\star x)=\pi(\exp(z))f_v(x).
\end{equation} 
This implies that the map 
\[
G\to \mathscr H\ ,\ g\mapsto \pi(g)f_v(x)
\] is analytic in a neighborhood of the identity, i.e., $f_v(x)\in\mathscr H^\omega$.
This completes the proof of (i).

(ii)  It suffices to prove the statement when $a\in\g g$. 
If $0<r<r_\circ$ then there exists an open set $W\sseq B_r\times B_r$ such that $\{0\}\times B_r\sseq W$ and for every 
$(z,z')\in W$ we have
$z\star z'\in B_r$.
Observe that the map
\begin{equation}
\label{eqn-defiofpsi}
\Psi:W\to\mathscr H\ ,\ \Psi(z,z')=f_v(z\star z')
\end{equation} is analytic in $W$. 
The map
\[
\C\to\mathscr H\ ,\ \zeta\mapsto f_v\big((\zeta \cdot a)\star x\big)
\]
is an analytic function of $\zeta$ in a neighborhood of the origin.  
From \eqref{eqnorbt} it follows that for every $x\in B_r$
we have 
\[
\dd\pi(a)(f_v(x))=\frac{\partial}{\partial \zeta}f_v\big((\zeta\cdot a)\star x\big)\Big|_{\zeta=0}.
\] 
From analyticity of the map $\Psi:W\to\mathscr H$ defined in \eqref{eqn-defiofpsi} it follows that
the map 
\[
B_r\to\mathscr H\ ,\ x\mapsto \frac{\partial}{\partial \zeta}f_v\big((\zeta\cdot a)\star x\big)\Big|_{\zeta=0}
\] 
is analytic in $B_r$.  This completes the proof of (ii).
\qedhere

\end{proof}

\section{Representations of Banach--Lie supergroups}

\label{sec-blsupergps}

Our next task is to define the notions of smooth and analytic unitary representations of a Banach--Lie supergroup.
The definitions 
are similar to the one given in \cite[Def. 2]{varadarajan} for finite dimensional Lie supergroups.

\begin{definition}
\label{defi-smoothanalytic}
Let $(G,\g g)$ be a Banach--Lie supergroup. A \emph{smooth unitary representation} 
(respectively, an \emph{analytic unitary representation}) of $(G,\g g)$ is a triple $(\pi,\rho^\pi,\mathscr H)$ satisfying the following properties.
\begin{enumerate}
\item[(i)] $(\pi,\mathscr H)$ is a smooth (respectively, analytic) unitary representation of $G$
on the $\mathbb Z_2$-graded Hilbert space $\mathscr H$ such that for every $g\in G$, the operator
$\pi(g)$ preserves the $\mathbb Z_2$-grading.
\item[(ii)] $\rho^\pi:\g g\to\End_\C(\mathscr B)$ is a representation of the Banach--Lie superalgebra $\g g$, where
$\mathscr B=\mathscr H^\infty$ (respectively, $\mathscr B=\mathscr H^\omega$).
\item[(iii)] $\rho^\pi(x)=\dd\pi(x)\big|_\mathscr B$ for every $x\in\g g_\eev$.
\item[(iv)] $e^{-\frac{\pi i}{4}}\rho^\pi(x)$ is a symmetric operator for every $x\in\g g_\ood$.
\item[(v)] Every element of the component group $G/G^\circ$ has a coset representative $g\in G$ such that $\pi(g)\rho^\pi(x)\pi(g)^{-1}=\rho^\pi(\Ad(g)x)$ for every $x\in\g g_\ood$.

\end{enumerate}
The cateogry of smooth (respectively, analytic) unitary representations of $(G,\g g)$ is denoted by 
$\fctr{Rep}^\infty(G,\g g)$ (respectively, $\fctr{Rep}^\omega(G,\g g)$).
\end{definition}

\begin{rmk}
If $G$ is connected then obviously Definition \ref{defi-smoothanalytic}(v) always holds
trivially. This point is the main
difference
between Definition \ref{defi-smoothanalytic} above and the definition given in \cite[Def. 2]{varadarajan} for finite dimensional Lie groups,
where it is assumed that 
\begin{equation}
\label{eqn-condin}
\pi(g)\rho^\pi(x)\pi(g)^{-1}=\rho^\pi(\Ad(g)x)\ \text{ for every $x\in\g g_\ood$ and every $g\in G$,}  
\end{equation}
while  the infinitesimal action is supposed to satisfy a weaker condition. 
Indeed Proposition \ref{proposition-conjugacy} below implies that for a (possibly disconnected) $G$
the equation 
\eqref{eqn-condin} follows from Definition \ref{defi-smoothanalytic}.  
\end{rmk}

We will need a slightly more general gadget than smooth and analytic unitary representations, and
we introduce it in the next definition.

\begin{definition}
\label{definition-of-pseudo} Let $(G,\g g)$ be a Banach--Lie supergroup. 
A \emph{pre-representation} of 
$(G,\g g)$ is a 4-tuple 
$
(\,\pi,\mathscr H,\mathscr B,\rho^\mathscr B\,)$
which satisfies
the following properties.
\begin{enumerate} 
\item[(i)] $(\pi,\mathscr H)$ is a  unitary representation of $G$  on the 
$\mathbb Z_2$-graded Hilbert space $\mathscr H=\mathscr H_\eev\oplus\mathscr H_\ood$. Moreover, $\pi(g)$ is an even operator for every $g\in G$.
\item[(ii)] $\mathscr B$ is a dense $\mathbb Z_2$-graded 
subspace of $\mathscr H$ such that \[
\mathscr B\sseq \bigcap_{x\in \g g_\eev}\mathcal D\big(\dd\pi(x)\big).
\]
\item[(iii)] $\rho^\mathscr B:\g g\to \End_\mathbb C(\mathscr B)$  
 is a representation of the Banach--Lie superalgebra $\g g$.
\item[(iv)]  If $x\in\g g_\eev$ then $\rho^\mathscr B(x)=\dd\pi(x)\big|_\mathscr B$ and 
$\rho^\mathscr B(x)$ is essentially skew-adjoint.
\item[(v)] If $x\in \g g_\ood$ then $e^{-\frac{\pi i}{4}}\rho^\mathscr B(x)$ is a symmetric operator. 

\item[(vi)] For every element of the component group $G/G^\circ$, there exists a coset representative
$g\in G$ such that $\pi(g)^{-1}\mathscr B\sseq \mathscr B$ and 
\[
\pi(g)\rho^\mathscr B(x)\pi(g)^{-1}=\rho^\mathscr B(\Ad(g)x)\text{ for every $x\in\g g_\ood$.}
\]


\end{enumerate}

\end{definition}
\begin{rmk}
(i) Observe that in Definition \ref{definition-of-pseudo}(iii) there are no continuity assumptions on the  
map $\rho^\mathscr B$. 

(ii) Definition \ref{definition-of-pseudo} implies that $\mathscr B\sseq\bigcap_{n\in \N}\mathcal D_n$ where
\begin{equation}
\mathcal D_n=\bigcap_{x_1,\ldots,x_n\in\g g_\eev}\mathcal D\big(\dd\pi(x_1)\cdots\dd\pi(x_n)\big).
\end{equation}
Consequently, it follows from \cite[Thm. 9.4]{khdifferentiable} that $\mathscr B\sseq\mathscr H^\infty$. Since
$\mathscr B$ is assumed to be dense in $\mathscr H$, the unitary representation $(\pi,\mathscr H)$ of $G$ is smooth.

(iii) When $G$ is connected, Definition \ref{definition-of-pseudo}(vi) always holds trivially.
The advantage of assuming the conjugacy invariance only for coset representatives (and not 
for every element of $G$) is that Theorem \ref{thm-stabili} will be applicable to the situations 
where $\mathscr B$ is not $G$-invariant. An example of this situation is the Fock space realization of the oscillator representation of 
$(\mathrm{OSp}_\mathrm{res}(\mathscr K),\widehat{\g{osp}}_\mathrm{res}(\mathscr K))$. See Section
\ref{sec-restricted} for further details.
\end{rmk}

For a Banach--Lie group $G$, the subspaces of smooth and analytic vectors of a unitary representation are $G$-invariant. Therefore
Lemma \ref{lem-criteriaforesa} has the following immediate consequence.

\begin{corollary}
\label{cor-prerep}
Let $(G,\g g)$ be a Banach--Lie supergroup, $(\pi,\rho^\pi,\mathscr H)$ be a smooth (respectively, analytic) unitary representation of $(G,\g g)$, and $\mathscr B=\mathscr H^\infty$ (respectively, $\mathscr B=\mathscr H^\omega$).
Then $(\pi,\mathscr H,\mathscr B,\rho^\pi)$ is a pre-representation of $(G,\g g)$.
\end{corollary}









\begin{lemma}
\label{lem-extension-of-operators}
Let $(G,\g g)$ be a Banach--Lie supergroup and
$\big(\pi,\mathscr H,\mathscr B,\rho^\mathscr B\big)$
be a pre-representation of $(G,\g g)$. Then the following statements hold.
\begin{enumerate}\item[(i)] For every $x\in \g g_\eev$ we have $\overline{\rho^\mathscr B(x)}=\dd\pi(x)$. In particular
$\mathscr H^\infty\sseq \mathcal D(\overline{\rho^\mathscr B(x)})$. 
\item[(ii)] For every $x\in \g g_\ood$ the operator $e^{-\frac{\pi i}{4}}\rho^\mathscr B(x)$ is essentially 
self-adjoint and $\overline{\rho^\mathscr B(x)}^2=\frac{1}{2}\dd\pi([x,x])$. In particular 
$\mathscr H^\infty\sseq \mathcal D(\overline{\rho^\mathscr B(x)})$. 
\end{enumerate}
\end{lemma}

\begin{proof}
(i) The statement follows directly from Definition \ref{definition-of-pseudo}(iv).

(ii) The statement follows from 
Lemma \ref{varalemma1}.
\qedhere
\end{proof}
\begin{nota}
Let $(\pi,\mathscr H,\mathscr B,\rho^\mathscr B)$ be a pre-representation of a Banach--Lie supergroup 
$(G,\g g)$. 
For every  $x=x_\eev+x_\ood\in\g g^\C$  we
define a linear operator $\tilde\rho^\mathscr B(x)$ on $\mathscr H$ 
with  $\mathcal D(\tilde\rho^\mathscr B(x))=\mathscr H^\infty$ as follows. If $x_\eev=a_\eev+ib_\eev$ and
$x_\ood=a_\ood+ib_\ood$ where $a_\eev,b_\eev\in \g g_\eev$ and $a_\ood,b_\ood\in\g g_\ood$ 
then for every $v\in\mathscr H^\infty$ we set
\[
\tilde\rho^\mathscr B(x)v=\overline{\rho^\mathscr B(a_\eev)}v+
i\overline{\rho^\mathscr B(b_\eev)}v
+\overline{\rho^\mathscr B(a_\ood)}v
+i\overline{\rho^\mathscr B(b_\ood)}v.
\]
\end{nota}

\begin{proposition}
\label{smooth-preserve}
Let $(G,\g g)$ be a Banach--Lie supergroup and
$\big(\pi,\mathscr H,\mathscr B,\rho^\mathscr B\big)$
be a pre-representation of $(G,\g g)$. Then 
$\tilde\rho^\mathscr B(x)\mathscr H^\infty\sseq \mathscr H^\infty$ 
for every $x\in\g g^\C$.

\end{proposition}
\begin{proof}
By Lemma \ref{lem-extension-of-operators}(i) and the definition of $\tilde\rho^\mathscr B$  
it suffices to prove the statement when $x\in \g g_\ood$.
As shown in \cite[Thm. 9.4]{khdifferentiable}, we have 
$\mathscr H^\infty=\bigcap_{n\in\N}\mathcal D_n$ where
\[
\mathcal D_n=\bigcap_{x_1,\ldots,x_n\in \g g_\eev}\mathcal D(\dd\pi(x_1)\cdots\dd\pi(x_n)).
\]
Therefore it is enough to prove that  
\begin{equation}
\label{inclusion-eq}
\overline{\rho^\mathscr B(x)}v\in\mathcal D_n
\text{ for every $x\in\g g_\ood$ and $v\in\mathscr H^\infty$}. 
\end{equation}
Let $y\in\g g_\eev$. For every $w\in\mathscr B$, using Lemma \ref{lem-extension-of-operators}
we can write
\begin{align*}
\langle \overline{\rho^\mathscr B(x)}v,\dd\pi(y)w\rangle&=
\langle\overline{\rho^\mathscr B(x)}v,\rho^\mathscr B(y)w\rangle \\
&=e^\frac{\pi i}{2}\langle v,\overline{\rho^\mathscr B(x)}\rho^\mathscr B(y)w\rangle\\
&=e^\frac{\pi i}{2}\langle v,\rho^\mathscr B(y)\rho^\mathscr B(x)w+
\rho^\mathscr B([x,y])w\rangle \\
&=e^\frac{\pi i}{2}
\langle v,\rho^\mathscr B(y)\rho^\mathscr B(x)w\rangle+
e^\frac{\pi i}{2}\langle v,\rho^\mathscr B([x,y])w\rangle  \\
&=\langle\overline{\rho^\mathscr B(x)}\dd\pi(y)v,w\rangle+
\langle \overline{\rho^\mathscr B([x,y])}v,w\rangle.
\end{align*}
It follows that the $\C$-linear functional
\[
\mathscr B\to\C\ ,\ w\mapsto \langle \overline{\rho^\mathscr B(x)}v,\dd\pi(y)w\rangle
\]
is continuous, i.e., 
$\overline{\rho^\mathscr B(x)}v\in\mathcal D\big((\dd\pi(y)\big|_\mathscr B)^*\big)$. 
Since $\dd\pi(y)\big|_\mathscr B=\rho^\mathscr B(y)$ is essentially skew-adjoint,
from Lemma
\ref{lem-extension-of-operators}(i) it follows that $(\dd\pi(y)\big|_\mathscr B)^*=-\dd\pi(y)$, i.e., 
$\overline{\rho^\mathscr B(x)}v\in\mathcal D(\dd\pi(y))$. This proves \eqref{inclusion-eq} for $n=1$.

For $n>1$ the proof of \eqref{inclusion-eq} can be completed by induction. Let $x_1,\ldots,x_n\in\g g_\eev$ and 
$v\in\mathscr H^\infty$. 
Using the induction hypothesis, for every $w\in\mathscr B$ we can write
\begin{align*}
\langle&\dd\pi(x_{n-1})\cdots\dd\pi(x_1)\overline{\rho^\mathscr B(x)}v,\dd\pi(x_n)w\rangle
=e^{(n-\frac{1}{2})\pi i}
\langle v,\rho^\mathscr B(x)\rho^\mathscr B(x_1)\cdots\rho^\mathscr B(x_n)w\rangle\\
&=e^{(n-\frac{1}{2})\pi i}\langle v,
\rho^\mathscr B([x,x_1])\rho^\mathscr B(x_2)\cdots\rho^\mathscr B(x_n)w+
\rho^\mathscr B(x_1)\rho^\mathscr B(x)\rho^\mathscr B(x_2)\cdots\rho^\mathscr B(x_n)w \rangle\\
&=\langle\dd\pi(x_n)\cdots\dd\pi(x_2)\overline{\rho^\mathscr B([x,x_1])}v,w\rangle
+\langle\dd\pi(x_n)\cdots\dd\pi(x_2)\overline{\rho^\mathscr B(x)}\dd\pi(x_1)v,w\rangle.
\end{align*}
An argument similar to the case $n=1$ proves that 
\[
\dd\pi(x_{n-1})\cdots\dd\pi(x_1)\overline{\rho^\mathscr B(x)}v\in\mathcal D(\dd\pi(x_n)).
\]
Consequently, $v\in\mathcal D_n$.
\qedhere
\end{proof}

\begin{proposition}
\label{lem-continu-of-overline}
Let $(G,\g g)$ be a Banach--Lie supergroup and
$\big(\pi,\mathscr H,\mathscr B,\rho^\mathscr B\big)$
be a pre-representation of $(G,\g g)$. Then the following statements hold.
\begin{enumerate}
\item[(i)] The map
\begin{equation}
\label{themapgtohinf}
\g g^\C \times\mathscr H^\infty\to\mathscr H^\infty\ ,\ (x,v)\mapsto \tilde\rho^\mathscr B(x)v
\end{equation}
is $\C$-bilinear.
\item[(ii)] If $x,y\in\g g^\C$ are homogeneous, then for every $v\in\mathscr H^\infty$ we have
\[
\tilde\rho^\mathscr B([x,y])v=\tilde\rho^\mathscr B(x)\tilde\rho^\mathscr B(y)v
-(-1)^{p(x)p(y)}\tilde\rho^\mathscr B(y)\tilde\rho^\mathscr B(x)v
\]
\item[(iii)] The map given in \eqref{themapgtohinf} is continuous. 
\end{enumerate} 
\end{proposition}

\begin{proof}
(i) By Lemma \ref{lem-extension-of-operators}(i) and the definition of $\tilde\rho^\mathscr B$ 
it is enough to prove that for every $v\in\mathscr H^\infty$ the map
\[
\g g_\ood\to\mathscr H^\infty\ ,\  x\mapsto\overline{\rho^\mathscr B(x)}v
\]
is $\R$-linear. 

Let $x\in\g g_\ood$ and $a\in \R$. Then the equality
\begin{equation}
\label{eqn-scalara}
\overline{\rho^\mathscr B(ax)}v=a\overline{\rho^\mathscr B(x)}v
\end{equation}
holds for every $v\in\mathscr B$, and therefore by Lemma \ref{lem-two-oper} it also holds for every $v\in\mathscr H^\infty$.
A similar reasoning proves that if $x,y\in \g g_\ood$ then for every $v\in\mathscr H^\infty$ we have
\begin{equation}
\label{eqn-sumxy}
\overline{\rho^\mathscr B(x+y)}v=\overline{\rho^\mathscr B(x)}v+\overline{\rho^\mathscr B(y)}v.
\end{equation} 

(ii) It suffices to prove the statement for $x,y\in\g g$. Depending on the parities of $x$ an $y$, there are four cases to consider, but the argument for all of them is esentially the same.
For example, if $x\in\g g_\eev$ and $y\in\g g_\ood$, then we define two operators $P_1$ and $P_2$ with domains
$\mathcal D(P_1)=\mathcal D(P_2)=\mathscr H^\infty$  as follows. If $v\in\mathscr H^\infty$ then we set
\[
P_1v=e^{-\frac{\pi i}{4}}\overline{\rho^\mathscr B([x,y])}v\ \text{ and }\ 
P_2v=e^{-\frac{\pi i}{4}}\Big(\overline{\rho^\mathscr B(x)}\,\overline{\rho^\mathscr B(y)}v-
\overline{\rho^\mathscr B(y)}\,\overline{\rho^\mathscr B(x)}v\Big).
\]
Then $P_1$ and $P_2$ are both symmetric, $P_1\big|_{\mathscr B}=P_2\big|_{\mathscr B}$, and by
Lemma \ref{lem-extension-of-operators}(ii) 
the operator $P_1\big|_{\mathscr B}$ is 
essentially self-adjoint. Lemma \ref{lem-two-oper} implies that $P_1=P_2$.

(iii) As in (i), it is enough to prove that the map
\begin{equation}
\label{eqn-conttinn}
\g g_\ood\times\mathscr H^\infty\to\mathscr H^\infty\;,\; x\mapsto\overline{\rho^\mathscr B(x)}v
\end{equation}
is continuous.
Let $v\in \mathscr H^\infty$ and $y\in\g g_\ood$. 
Setting $P_1=e^{-\frac{\pi i}{4}}\overline{\rho^\mathscr B(y)}$ 
and $P_2=\frac{1}{2}e^{-\frac{\pi i}{2}}\dd\pi([y,y])$
in Lemma \ref{lem-helpshorten} we obtain
\begin{align*}
\|\overline{\rho^\mathscr B(y)}v\|&\leq\frac{1}{\sqrt{2}}\|v\|^\frac{1}{2}\cdot \|\dd\pi([y,y])v\|^\frac{1}{2}\\
&\leq \frac{1}{\sqrt{2}}\nrm_0(v)^\frac{1}{2}\nrm_1(v)^\frac{1}{2}\cdot\|[y,y]\|^\frac{1}{2}
\leq \frac{1}{2\sqrt{2}}\big(\nrm_0(v)+\nrm_1(v)\big)\cdot \|y\|
\end{align*}
from which it follows that 
\begin{equation}
\label{eestone}
\nrm_0\big(\overline{\rho^\mathscr B(y)}v\big)\leq \frac{1}{2\sqrt{2}}\big(\nrm_0(v)+\nrm_1(v)\big)\cdot\|y\|.
\end{equation}
If $x\in\g g_\eev$ satisfies $\|x\|\leq 1$ then $\|[x,y]\|\leq \|y\|$ and using \eqref{eestone} we obtain
\begin{align*}
\|\dd\pi(x)\overline{\rho^\mathscr B(y)}v\|&=\nrm_0\big(\overline{\rho^\mathscr B(x)}\,
\overline{\rho^\mathscr B(y)}v\big)
\leq \nrm_0\big(\overline{\rho^\mathscr B(y)}\,\overline{\rho^\mathscr B(x)}v\big)+
\nrm_0\big(\overline{\rho^\mathscr B([x,y])}v\big)\\
&\leq \frac{1}{2\sqrt{2}}\big(\nrm_0(\dd\pi(x)v)+\nrm_1(\dd\pi(x)v)\big)\cdot \|y\|+\frac{1}{2\sqrt{2}}\big(\nrm_0(v)+\nrm_1(v)\big)\cdot\|[x,y]\|\\
&\leq \frac{1}{2\sqrt{2}}\big(\nrm_1(v)+\nrm_2(v)\big)\cdot\|y\|+\frac{1}{2\sqrt{2}}
\big(\nrm_0(v)+\nrm_1(v)\big)\cdot\|y\|\\
&=\frac{1}{2\sqrt{2}}\big(\nrm_0(v)+2\nrm_1(v)+\nrm_2(v)\big)\cdot\|y\|
\end{align*}
from which it follows that 
\begin{equation}
\label{esttwo}
\nrm_1\big(\overline{\rho^\mathscr B(y)}v\big)\leq\frac{1}{2\sqrt{2}}\big(\nrm_0(v)+2\nrm_1(v)+\nrm_2(v)\big)\cdot\|y\|.
\end{equation}
More generally, if $x_1,\ldots,x_n\in\g g_\eev$ satisfy $\|x_1\|\leq 1,\ldots,\|x_n\|\leq 1$ then we can use the equality
\begin{align*}
\dd\pi(x_1)\cdots\dd\pi(x_n)\overline{\rho^\mathscr B(y)}v
&=\overline{\rho^\mathscr B(y)}\,\overline{\rho^\mathscr B(x_1)}\cdots\overline{\rho^\mathscr B(x_n)}v\\
&+\sum_{i=1}^n\overline{\rho^\mathscr B(x_1)}\cdots\overline{\rho^\mathscr B(x_{i-1})}
\,\overline{\rho^\mathscr B([x_i,y])}\,\overline{\rho^\mathscr B(x_{i+1})}\cdots
\overline{\rho^\mathscr B(x_n)}v
\end{align*}
to prove by induction on $n$ that
\begin{equation}
\label{iineqconvinf}
\nrm_n\big(\rho^\pi(y)v\big)\leq \frac{1}{2\sqrt{2}}\|y\|\cdot\sum_{k=0}^{n+1}\binom{n+1}{k}\nrm_k(v)\ \text{ for every }n\geq 0.
\end{equation}
From \eqref{iineqconvinf} the continuity of \eqref{eqn-conttinn} follows immediately.
\qedhere
\end{proof}

Our next goal is to prove Proposition \ref{proposition-conjugacy} below. The proof of the latter proposition is based on
a subtle lemma from \cite[Chap. 3]{jorgenson}. We use the lemma in the form given in 
\cite{merigon}.
\begin{lemma}
\label{lem-jorgenson}
Let $A$ and $B$ be two linear operators on a complex Hilbert space $\mathscr H$ and $\mathscr D$ be a dense subspace
of $\mathscr H$ with the following properties.
\begin{enumerate}
\item[(i)] $\mathcal D(A)=\mathcal D(B)=\mathscr D$. 
\item[(ii)] $A$ is essentially skew-adjoint.
\item[(iii)] $A\mathscr D\sseq\mathscr D$.
\item[(iv)] $e^{tA}\mathscr D\sseq \mathscr D$ for every $t\in\R$.
\item[(v)] $B$ is closable.
\end{enumerate}
Let $v\in\mathscr D$ be such that the map 
\[\R\to\mathscr H\;,\; t\mapsto BAe^{tA}v
\]
is continuous. Then the map 
\[
\R\to\mathscr H\ ,\ t\mapsto Be^{tA}v
\] 
is differentiable and $\displaystyle\frac{d}{dt}(Be^{tA}v)=BAe^{tA}v$.
\end{lemma}
\begin{proof}
See \cite[Lem. 5]{merigon}.
\end{proof}

\begin{proposition}
\label{proposition-conjugacy}
Let $(G,\g g)$ be a Banach--Lie supergroup where $G$ is connected and
$\big(\pi,\mathscr H,\mathscr B,\rho^\mathscr B\big)$
be a pre-representation of $(G,\g g)$. Then for every $g\in G$, every $x\in\g g_\ood^\C$, and every 
$v\in\mathscr H^\infty$ we 
have 
\begin{equation}
\label{eqn-conjugacy}
\pi(g)\tilde\rho^\mathscr B(x)\pi(g)^{-1}v=\tilde\rho^\mathscr B(\Ad(g)x)v.
\end{equation}
\end{proposition}

\begin{proof}
Fix $y\in\g g_\eev$ and set $A(s)=\pi(\exp((1-s)y))$ for every $s\in\R$. 
For every $v\in\mathscr H^\infty$ and every $s\in\R$ define 
\[
K(s):\g g_\ood^\C\to \mathscr H\,,\, K(s)x=A(s)
\tilde\rho^\mathscr B(x)
A(s)^{-1}v.
\]
If $g=\exp(y)$ and $\gamma(s)=\Ad(\exp(sy))x$ for every $s\in\R$  then the left hand side of \eqref{eqn-conjugacy}  is equal to $K(0)\gamma(0)$ and the right hand side 
of \eqref{eqn-conjugacy} is equal to $K(1)\gamma(1)$. Therefore it suffices to prove that 
the map
$s\mapsto K(s)\gamma(s)$ is constant.

By Proposition \ref{lem-continu-of-overline}, for every $t\in\R$ the operator $K(t)$ is
bounded. From continuity of the maps \eqref{ghinfhinf} and \eqref{gghinfhinf} it follows that 
the map 
\[
\R\to\mathscr H\,,\, s\mapsto \tilde\rho^\mathscr B(x)\tilde\rho^\mathscr B(y)A(s)^{-1}v
\] 
is continuous as well. It follows from
Lemma \ref{lem-jorgenson} that the map
\[
\eta:\R\mapsto\mathscr H\, ,\, \eta(s)=\tilde\rho^\mathscr B(x)A(s)^{-1}v
\]
is differentiable, and $\eta'(s)=\tilde\rho^\mathscr B(x)\tilde\rho^\mathscr B(y)A(s)^{-1}v$. 

Next we show that the map 
\begin{equation}
\label{map-k(s)x}
\R\to\mathscr H\, ,\,
s\mapsto K(s)x
\end{equation}
is differentiable and we compute its derivative. 
Observe that
\begin{align*}
\frac{d}{ds}(K(s)x)&=\frac{d}{ds}(A(s)\eta(s))=\lim_{h\to 0}\frac{1}{h}
\big(A(s+h)\eta(s+h)-A(s)\eta(s)\big)
\end{align*}
and 
\begin{align}
\label{eqn-as+h}
\notag
\frac{1}{h}&\big(A(s+h)\eta(s+h)-A(s)\eta(s)\big)\\
&=A(s+h)\Big(\frac{1}{h}\big(\eta(s+h)-\eta(s)\big)\Big)+
\frac{1}{h}\big(A(s+h)\eta(s)-A(s)\eta(s)\big).
\end{align}
The first term in \eqref{eqn-as+h} can be written as 
\[
A(s+h)\Big(\frac{1}{h}\big(\eta(s+h)-\eta(s)\big)-\eta'(s)\Big)+A(s+h)\eta'(s).
\]
Since $\|A(s+h)\|_\mathrm{Op}=1$, when $h\to 0$ we obtain 
\[
A(s+h)\Big(\frac{1}{h}\big(\eta(s+h)-\eta(s)\big)-\eta'(s)\Big)\to 0
\text{ and }A(s+h)\eta'(s)\to A(s)\eta'(s).
\]
Since 
$\eta(s)\in\mathscr H^\infty$, as $h\to 0$ the second term in \eqref{eqn-as+h} converges to 
$
-A(s)\tilde\rho^\mathscr B(y)\eta(s)
$.
It follows that 
\begin{equation}
\label{derivofksx}
\frac{d}{ds}\big(K(s)x\big)=A(s)\eta'(s)-A(s)\tilde\rho^\mathscr B(y)\eta(s)=A(s)
[\tilde\rho^\mathscr B(x),\tilde\rho^\mathscr B(y)]A(s)^{-1}v.
\end{equation}
To complete the proof, it suffices to show that $\frac{d}{ds}(K(s)\gamma(s))=0$ for every $s\in\R$.
We have
\begin{align}
\label{eqn-ks+hfinalapprox}
\notag \frac{1}{h}&\big(K(s+h)\gamma(s+h)-K(s)\gamma(s)\big)\\
&=K(s+h)\Big(\frac{1}{h}\big(\gamma(s+h)-\gamma(s)\big)\Big)+
\frac{1}{h}\big(K(s+h)\gamma(s)-K(s)\gamma(s)\big).
\end{align}
The first term in \eqref{eqn-ks+hfinalapprox} can be written as 
\[
K(s+h)\Big(\frac{1}{h}\big(\gamma(s+h)-\gamma(s)\big)-\gamma'(s)\Big)+K(s+h)\gamma'(s)
\]
Differentiability of the map given in \eqref{map-k(s)x} implies that it is continuous, and in particular 
if $I$ is a compact interval containing $s$, then $\sup_{t\in I}\|K(t)x\|<\infty$. The Banach--Steinhaus Theorem
implies that $\sup_{t\in I}\|K(t)\|_\mathrm{Op}<\infty$. It follows that as $h\to 0$ the first term in 
\eqref{eqn-ks+hfinalapprox} converges to 
\[
K(s)\gamma'(s)=A(s)\tilde\rho^\mathscr B([y,\gamma(s)])A(s)^{-1}v.
\]
By \eqref{derivofksx}, the second term in \eqref{eqn-ks+hfinalapprox} converges to
\[
A(s)[\tilde\rho^\mathscr B(\gamma(s)),\tilde\rho^\mathscr B(y)]A(s)^{-1}v.
\]
It follows that 
\begin{align*}
\frac{d}{ds}(K(s)\gamma(s))=A(s)\tilde\rho^\mathscr B([y,\gamma(s)])A(s)^{-1}v
+A(s)[\tilde\rho^\mathscr B(\gamma(s)),\tilde\rho^\mathscr B(y)]A(s)^{-1}v=0.
\end{align*}
\end{proof}

\begin{proposition}
\label{mainthmmm}
Let $(G,\g g)$ be a Banach--Lie supergroup  and
$\big(\pi,\mathscr H,\mathscr B,\rho^\mathscr B\big)$
be a pre-representation of $(G,\g g)$. 
Then there exists an $r_\circ>0$ such that for every positive $r<r_\circ$ and 
every $x\in\g g^\C$ we have 
$\tilde\rho^\mathscr B(x)\mathscr H^{\omega,r}\sseq \mathscr H^{\omega,r}$.
In particular, for every $x\in \g g^\C$ we have $\tilde\rho^\mathscr B(x)\mathscr H^\omega\sseq\mathscr H^\omega$.


\end{proposition}

\begin{proof}
It suffices to prove the statement when $x$ is a homogeneous element of $\g g$. 
We give the argument for $x\in\g g_\ood$. The argument for the case
$x\in\g g_\eev$ is analogous.

Let $r_\circ$ be the constant 
obtained from Lemma \ref{lem-analyfv} and $v\in\mathscr H^{\omega,r}$ where $0<r<r_\circ$. 
Recall that
\[
B_r=\{\ y\in\g g_\eev^\C \ :\ \|y\|<r\ \}
\]
The map $\displaystyle y\mapsto f_v(y)=\sum\frac{1}{n!}\dd\pi(x)^nv$ 
of Lemma \ref{convoncplx} is analytic in $B_r$. 
Since 
\[
\sum_{n=0}^\infty\frac{1}{n!}\|\ad_y^n(x)\|\leq\sum_{n=0}^\infty \frac{1}{n!}\|y\|^n\cdot\|x\|<\infty,
\]
Theorem \ref{firstthmbosi}(ii)
implies that the map
\begin{equation}
\label{eqn-adjointact}
\g g_\eev^\C\to\g g_\ood^\C\,,\,  y\mapsto e^{\ad_y}x=\sum_{n=0}^\infty\frac{1}{n!}\ad_y^n(x)
\end{equation}
is analytic in $\g g_\eev^\C$. 

Consider the map
\[
\varphi: \g g_\eev^\C\times B_r\to\mathscr H\,,\,
\varphi(y,z)=\tilde\rho^\mathscr B(e^{\ad_y}x)f_v(z).
\] 
Note that if $z\in B_r$ then by Lemma \ref{lem-analyfv}(i) we have $f_v(z)\in\mathscr H^\omega$ and therefore $\varphi$  is well-defined. 
Our next goal is to prove that $\varphi$ is separately analytic. 

Fix $z\in B_r$ and let $w=f_v(z)$. 
We have $w\in\mathscr H^\omega$. Proposition \ref{lem-continu-of-overline} 
implies that the map
\[
\g g_\ood^\C\to \mathscr H\,,\, u\mapsto \tilde\rho^\mathscr B(u)w
\]
is $\C$-linear and continuous. Therefore analyticity of 
the map given in \eqref{eqn-adjointact}
implies that the map $y\mapsto\varphi(y,z)$ is analytic in $\g g_\eev^\C$.

Next we prove analyticity in $B_r$ of the map $z\mapsto \varphi(y,z)$.
Fix $y\in\g g_\eev^\C$ and set $\tilde{y}=\sum_{n=0}^\infty\frac{1}{n!}\ad_y^n(x)$. Writing 
$\tilde y=a+ib$ where $a,b\in\g g_\ood$ and using $\C$-linearity of $\tilde\rho^\mathscr B$, 
it turns out that it suffices to 
prove analyticity in $B_r$  of the maps \[
z\mapsto \tilde\rho^\mathscr B(a)f_v(z)\ \text{ and }\
z\mapsto \tilde\rho^\mathscr B(b)f_v(z). 
\]
The argument for both cases is the same, and we only give it for the first case. 
Lemma \ref{lem-analyfv}(ii) implies that the map 
\[
B_r\mapsto \mathscr H\,,\, z\mapsto\tilde\rho^\mathscr B([a,a])f_v(z)
\]
is analytic in $B_r$.
The operator $T=e^{-\frac{\pi i}{4}}\overline{\tilde\rho^\mathscr B(a)}$ is self-adjoint and spectral theory
implies that $T(I+T^2)^{-1}$ is a bounded operator. Using Proposition \ref{lem-continu-of-overline}(ii)
we can write
\begin{align*}
\tilde\rho^\mathscr B(a)f_v(z)&=e^{\frac{\pi i}{4}}T(I+T^2)^{-1}(I+T^2)f_v(z)\\
&=e^{\frac{\pi i}{4}}T(I+T^2)^{-1}
\big(f_v(z)+\frac{1}{2}e^{-\frac{\pi i}{2}}\tilde\rho^\mathscr B([a,a])f_v(z)\big)
\end{align*}
from which it follows that the map $z\mapsto \tilde\rho^\mathscr B(a)f_v(z)$ is analytic in $B_r$.

By Theorem \ref{thirdthmbosi}, the separately analytic map $\varphi(y,z)$ is analytic. In particular, 
the map
\[
B_r\to\mathscr H\, ,\, y\mapsto \varphi(y,y)
\]
is analytic in $B_r$. 

Let $y\in B_r\cap\g g_\eev$. By \eqref{eqn-piexpxv} we have
$f_v(y)=\pi(\exp(y))v$.
Write $x=x'+ix''$ where $x',x''\in\g g_\ood$, and set
$g=\exp(y)$. Using Proposition \ref{proposition-conjugacy} we obtain 
\begin{align*}
\varphi(y,y)&= 
\tilde\rho^\mathscr B(e^{\ad_y}x)f_v(y)\\
&=
\tilde\rho^\mathscr B\big(\Ad\big(\exp(y)\big)x'\big)f_v(y)+i\tilde\rho^\mathscr B\big(\Ad\big(\exp(y)\big)x''\big)f_v(y)\\
&
=
\pi(g)\tilde\rho^\mathscr B(x')\pi(g)^{-1}f_v(y)+i\pi(g)\tilde\rho^\mathscr B(x'')\pi(g)^{-1}f_v(y)\\
&=\pi(g)\tilde\rho^\mathscr B(x)\pi(g)^{-1}\pi(g)v=\pi(g)\tilde\rho^\mathscr B(x)v.
\end{align*}
From Lemma \ref{analyticBr} it follows that $\tilde\rho^\mathscr B(x)v\in\mathscr H^{\omega,r}$.
\qedhere
\end{proof}

We can now prove the first main theorem of this article, which states that every pre-representation of a Banach--Lie group corresponds to a unique unitary representation.

\begin{theorem}
\label{thm-stabili}
\textbf{\upshape(Stability Theorem)}
Let  $(\pi,\mathscr H,\mathscr B,\rho^\mathscr B)$ be a  
pre-representation 
of a Banach--Lie supergroup $(G,\g g)$. Then the following statements hold.
\begin{enumerate}
\item[(i)] There exists a unique map 
\[
\rho^\pi:\g g\to\End_\C(\mathscr H^\infty)
\]
such that 
$\rho^\pi(x)\big|_{\mathscr B}=\rho^\mathscr B(x)$ and
$(\pi,\rho^\pi,\mathscr H)$ is a smooth unitary representation of $(G,\g g)$. 
\item[(ii)] If $(\pi,\mathscr H)$ is an analytic unitary representation of $G$, then there exists a unique
map \[
\rho^\pi:\g g\to\End_\C(\mathscr H^\omega)
\] 
such that 
$\overline{\rho^\pi(x)}\big|_{\mathscr B}=\rho^\mathscr B(x)$ and
$(\pi,\rho^\pi,\mathscr H)$ is an analytic unitary representation of $(G,\g g)$. 

\end{enumerate}
\end{theorem}
\begin{proof}
(i) To prove the existence of $\rho^\pi$, we set $\rho^\pi(x)=\tilde\rho^\mathscr B(x)$ for every $x\in\g g$. Proposition \ref{smooth-preserve} implies that 
$\tilde\rho^\mathscr B(x)\in\End_\C(\mathscr H^\infty)$. 
To prove the conjugacy invariance relation 
of Definition
\ref{defi-smoothanalytic}(v), for every element of $G/G^\circ$ we take a coset representative
$g\in G$ which satisfies the condition of Definition \ref{definition-of-pseudo}(vi), and use Lemma 
\ref{lem-two-oper} with $P_1=e^{-\frac{\pi i}{4}}\pi(g)\tilde\rho^\mathscr B(x)\pi(g)^{-1}$, 
$P_2=e^{-\frac{\pi i}{4}}\tilde\rho^\mathscr B(\Ad(g)x)$, and $\mathscr L=\mathscr B$.

To prove uniqueness, it suffices to show that if $(\pi,\rho^\pi,\mathscr H)$ is a 
smooth unitary representation such that for every $x\in\g g$ we have
$\rho^\pi(x)\big|_{\mathscr B}=\rho^\mathscr B(x)$ 
then for every $x\in\g g$ we have 
\begin{equation}
\label{eqn-rhopibigi}
\rho^\pi(x)\big|_{\mathscr H^\infty}=\tilde\rho^\mathscr B(x).
\end{equation}
It suffices to prove \eqref{eqn-rhopibigi} when $x$ is 
homogeneous. If $x\in\g g_\ood$ then 
by Lemma \ref{lem-extension-of-operators}(ii) the operator 
$e^{-\frac{\pi i}{4}}\tilde\rho^\mathscr B(x)\big|_\mathscr B$ is essentially self-adjoint. Therefore
\eqref{eqn-rhopibigi} follows from setting 
$P_1=e^{-\frac{\pi i}{4}}\rho^\pi(x)\big|_{\mathscr H^\infty}$, $P_2=e^{-\frac{\pi i}{4}}\tilde\rho^\mathscr B(x)$, and 
$\mathscr L=\mathscr B$ in
Lemma \ref{lem-two-oper}. The argument for $x\in\g g_\eev$ is similar.

(ii) Existence follows from (i) and Proposition \ref{mainthmmm}. The proof of uniqueness is similar to the one given
in (i).
\end{proof}

Let $\fctr F:\fctr{Rep}^\omega(G,\g g)\to\fctr{Rep}^\infty(G,\g g)$ be the functor defined by
\[
(\pi,\rho^\pi,\mathscr H)\mapsto (\pi,\tilde\rho^{\mathscr H^\omega},\mathscr H).
\]
A morphism between two objects of $\fctr{Rep}^\omega(G,\g g)$ will automatically become a morphism between their images under 
$\fctr F$ in $\fctr{Rep}^\infty(G,\g g)$, and in fact $\fctr F$ is fully faithful. Let $\fctr{Rep}_a^\infty(G,\g g)$ denote the full subcategory of $\fctr{Rep}^\infty(G,\g g)$ whose objects are smooth unitary representations $(\pi,\rho^\pi,\mathscr H)$ of $(G,\g g)$ such that $(\pi,\mathscr H)$ is an analytic unitary representation of $G$.
\begin{corollary}
\label{cor-repa}
The functor $\fctr F$ is an isomorphism of the categories  $\fctr{Rep}^\omega(G,\g g)$ and $\fctr{Rep}^{\infty}_a(G,\g g)$.
\end{corollary}
\begin{proof}
Follows immediately from the uniqueness statement of Theorem \ref{thm-stabili}(ii).\qedhere
\end{proof}

\begin{rmk}
A natural question that arises from Corollary \ref{cor-repa} is whether or not  $\fctr F$
is an isomorphism between the categories 
$\fctr{Rep}^\omega(G,\g g)$ and $\fctr{Rep}^\infty(G,\g g)$. 
We answer this question negatively in Appendix \ref{exampleanalytic} by giving an example of a Banach--Lie group 
$G$ with a smooth unitary representation $(\pi,\mathscr H)$ which does not have any analytic vectors.
\end{rmk}

Let $(G,\g g)$ be a Banach--Lie supergroup and $(H,\g h)$ be an integral subsupergroup of $(G,\g g)$. 
One can obtain restriction functors
\[
\fctr{Res}^\infty:\fctr{Rep}^\infty(G,\g g)\to\fctr{Rep}^\infty(H,\g h)
\ \text{ and }\
\fctr{Res}^\omega:\fctr{Rep}^\omega(G,\g g)\to\fctr{Rep}^\omega(H,\g h)
\]
as follows. If $(\pi,\rho^\pi,\mathscr H)$ is a smooth (respectively, analytic) unitary representation of
$(G,\g g)$, then we set $\mathscr L=\mathscr H^\infty$ (respectively, $\mathscr L=\mathscr H^\omega$). 
Observe that $(\pi\big|_H,\mathscr H,\mathscr L,\rho^\mathscr L)$ is a pre-representation of $(H,\g h)$.
The functor
$\fctr{Res}^\infty$ (respectively, $\fctr{Res}^\omega$) maps $(\pi,\rho^\pi,\mathscr H)$ to the unique  smooth (respectively, analytic) unitary representation of $(H,\g h)$ which corresponds to this pre-representation. (The existence and uniqueness of this unitary representation follows from  Theorem \ref{thm-stabili}.)
In conclusion, we have proved the following theorem.

\begin{theorem}
\label{thm-restfunct}
\textbf{\upshape(Restriction Theorem)}
Let  $(G,\g g)$ be a Banach--Lie supergroup, $(H,\g h)$ be an integral subsupergroup of $(G,\g g)$,
$(\pi,\mathscr H, \rho^\pi)$ be a smooth (respectively, analytic) unitary representation of $(G,\g g)$, and 
$\mathscr L=\mathscr H^\infty$ (respectively, $\mathscr L=\mathscr H^\omega$).
Then there exists a unique smooth (respectively, analytic) unitary representation 
$(\sigma,\mathscr H,\rho^\sigma)$ of $(H,\g h)$ with the following properties.
\begin{itemize}
\item[(i)] For every $h\in H$ we have $\pi(h)=\sigma(h)$.
\item[(ii)] For every $x\in \g h$ we have $\rho^\sigma(x)\big|_\mathscr L=\rho^\pi(x)$.
\end{itemize}
\end{theorem}

\section{The Oscillator representation of $(\widehat{\mathrm{OSp}}_\mathrm{res}(\mathscr K),\widehat{\g{osp}}_\mathrm{res})$}
\label{sec-restricted}

In this section we show that the oscillator representation of 
the restricted orthosymplectic Banach--Lie supergroup is an analytic unitary representation in the sense of 
Definition \ref{defi-smoothanalytic}. To simplify the presentation, we have omitted some of the tedious computations. They can be done using the method of \cite[Sec. 9]{neebconflu}.

Let $\mathscr K=\mathscr K_\eev\oplus\mathscr K_\ood$ be a $\mathbb Z_2$-graded complex Hilbert space.
For simplicity we assume that both $\mathscr K_\eev$ and $\mathscr K_\ood$ are infinite dimensional. The case where one or both of these spaces are finite dimensional is similar.

 We denote the inner
product of $\mathscr K$ by $\langle\cdot,\cdot\rangle$. By restriction of scalars we can also consider $\mathscr K$ as a real Hilbert space.

Let $J_+:\mathscr K\to\mathscr K$ denote multiplication by $\sqrt{-1}$, and $J_-:\mathscr K\to\mathscr K$ be the
map defined by 
\[
J_-v=-(-1)^{p(v)}\sqrt{-1}\,v 
\ \text{ for every homogeneous }v\in\mathscr K.
\]
In the following, both $J_+$ and $J_-$ will be considered as $\R$-linear maps.

If $A:\mathscr K\to\mathscr K$ is an $\R$-linear map, then 
$AJ_+-J_+A$ is $\C$-conjugate~linear.
The space of Hilbert--Schmidt $\R$-linear maps on 
$\mathscr K$ (respectively, on $\mathscr K_s$ where $s\in\{\eev,\ood\}$) is denoted by $\mathrm{HS}(\mathscr K)$
(respectively, by $\mathrm{HS}(\mathscr K_s)$).
The group of bounded invertible $\R$-linear maps on $\mathscr K$ (respectively, on $\mathscr K_s$ where $s\in\{\eev,\ood\}$)
is denoted by $\mathrm{GL}(\mathscr K)$ 
(respectively, by $\mathrm{GL}(\mathscr K_s)$).
We set
\[
\mathrm{GL}_\mathrm{res}(\mathscr K)=\{\,T\in\mathrm{GL}(\mathscr K)\ :\ TJ_+-J_+T\in\mathrm{HS}(\mathscr K)\,\}.
\]
The groups $\mathrm{GL}_\mathrm{res}(\mathscr K_s)$, where $s\in\{\eev,\ood\}$, are defined similarly.

Let $\End_\R(\mathscr K)=\End_\R(\mathscr K)_\eev\oplus\End_\R(\mathscr K)_\ood$ denote the superalgebra of bounded $\R$-linear maps 
on $\mathscr K$. Every $T\in\End_\R(\mathscr K)$ can be written in a unique way as 
$T=T_\mathrm{lin}+T_\mathrm{conj}$ where
$T_\mathrm{lin}$ is $\C$-linear and $T_\mathrm{conj}$ is $\C$-conjugate linear. In fact we have
\[
T_\mathrm{lin}=\frac{1}{2}(T-J_+TJ_+)\,\text{ and }\,T_\mathrm{conj}=\frac{1}{2}(T+J_+TJ_+).
\]

The inner product of the complex Hilbert space $\mathscr K$ yields
$\R$-bilinear forms 
$(\cdot,\cdot)_\eev$ on $\mathscr K_\eev$ and $(\cdot,\cdot)_\ood$  on $\mathscr K_\ood$ 
defined by
\[
(v,w)_\eev=\mathrm{Re}\langle v,w\rangle\text{ for every }v,w\in\mathscr K_\eev
\]
and 
\[
(v,w)_\ood=\mathrm{Im}\langle v,w\rangle\text{ for every }v,w\in \mathscr K_\ood.
\]
As a result, we obtain an $\R$-bilinear form $(\cdot,\cdot)=(\cdot,\cdot)_\eev\oplus(\cdot,\cdot)_\ood$ on $\mathscr K=\mathscr K_\eev\oplus\mathscr K_\ood$.

The \emph{restricted orthogonal group} $\mathrm{O}_\mathrm{res}(\mathscr K_\eev)$ is defined by
\[
\mathrm{O}_\mathrm{res}(\mathscr K_\eev)=
\{\, T \in \mathrm{GL}_\mathrm{res}(\mathscr K_\eev)\ :\ (Tv,Tw)_\eev=(v,w)_\eev \text{ for every }v,w\in \mathscr K_\eev\,\}.
\] 
Observe that 
\[\g{o}_\mathrm{res}(\mathscr K_\eev)
=
\Lie(\mathrm{O}_\mathrm{res}(\mathscr K_\eev))=
\{\,T\in\g{gl}_\mathrm{res}(\mathscr K_\eev)\ :\ (Tv,w)_\eev+(v,Tw)_\eev=0\,\},
\]
where \[
\g{gl}_\mathrm{res}(\mathscr K_\eev)=\{\,T\in\End_\R(\mathscr K_\eev)\ :\ T_\mathrm{conj}\in\mathrm{HS}(\mathscr K_\eev)\,\}.
\]
The definitions of the \emph{restricted symplectic group} $\mathrm{Sp}_\mathrm{res}(\mathscr K_\ood)$ and its Lie algebra 
$\g{sp}_\mathrm{res}(\mathscr K_\ood)=\Lie(\mathrm{Sp}_\mathrm{res}(\mathscr K_\ood))$ are analogous.

The Banach--Lie superalgebra
$\g{osp}_\mathrm{res}(\mathscr K)$ is the subspace of $\End_\R(\mathscr K)$ spanned by elements 
$T\in\End_\R(\mathscr K)_\eev\cup\End_\R(\mathscr K)_\ood$
with the following properties.
\begin{itemize}
\item[(i)] For every $v,w\in\mathscr K$, we have  $(Tv,w)+(-1)^{p(T)p(v)}(v,Tw)=0$.

\item[(ii)] $T_\mathrm{conj}\in\mathrm{HS}(\mathscr K)$.

\end{itemize}
The norm $\|\cdot\|$ on $\g{osp}_\mathrm{res}(\mathscr K)$ is given as follows. For every $T\in\g{osp}_\mathrm{res}(\mathscr K)$, 
we set
\[
\|T\|'=\|\,T_\mathrm{lin}\,\|_\mathrm{Op}+\|\,T_\mathrm{conj}\,\|_\mathrm{HS}
\]
where $\|\cdot\|_\mathrm{Op}$ denotes the operator norm of a $\C$-linear operator on the (complex) Hilbert space $\mathscr K$
and $\|\cdot\|_\mathrm{HS}$ denotes the Hilbert-Schmidt norm of an $\R$-linear operator on the (real) Hilbert space $\mathscr K$.
One can prove that $\|\cdot\|'$ is continuous, and therefore by a suitable scaling one obtains a norm $\|\cdot\|$ which satisfies 
\eqref{eqn-normineq}.

The
\emph{restricted orthosymplectic Banach--Lie supergroup} associated to $\mathscr K$ is the 
Banach--Lie supergroup $(\,\mathrm{OSp}_\mathrm{res}(\mathscr K),\g{osp}_\mathrm{res}(\mathscr K)\,)$ where 
\[
\mathrm{OSp}_\mathrm{res}(\mathscr K)=\mathrm O_\mathrm{res}(\mathscr K_\eev)\times\mathrm{Sp}_\mathrm{res}(\mathscr K_\ood).
\]

It is known \cite{segal} that to  realize the spin representation of $\mathrm O_\mathrm{res}(\mathscr K_\eev)$ 
or the metaplectic representation of $\mathrm{Sp}_\mathrm{res}(\mathscr K_\ood)$ one needs to pass to 
certain central extensions 
$\widehat{\mathrm{O}}_\mathrm{res}(\mathscr K_\eev)$ and $\widehat{\mathrm{Sp}}_\mathrm{res}(\mathscr K_\ood)$
which are also Banach--Lie groups \cite[Sec.~9]{neebconflu}. This leads to a Banach--Lie supergroup 
$(\widehat{\mathrm{OSp}}_\mathrm{res}(\mathscr K),\widehat{\g{osp}}_\mathrm{res}(\mathscr K))$ where
the Banach--Lie superalgebra $\widehat{\g{osp}}_\mathrm{res}(\mathscr K)$ is the central extension of
$\g{osp}_\mathrm{res}(\mathscr K)$
corresponding to the cocycle
\[
\omega:\g{osp}_\mathrm{res}(\mathscr K)\times\g{osp}_\mathrm{res}(\mathscr K)\to\R
\]
which can be uniquely identified by the following properties.
\begin{itemize}
\item[(i)] If $A,B\in\g{osp}_\mathrm{res}(\mathscr K)$ have different parity then $\omega(A,B)=0$.
\item[(ii)] If $A,B\in\g{osp}_\mathrm{res}(\mathscr K)_\eev$ then $\omega(A,B)=-\frac{1}{2}\,\tr(J_+A_\mathrm{conj}B_\mathrm{conj})$.
\item[(iii)] If $A,B\in\g {osp}_\mathrm{res}(\mathscr K)_\ood$ then $\omega(A,B)=-\frac{1}{2}\,\tr(J_-A_\mathrm{conj}B_\mathrm{conj})$.
\end{itemize}
In other words we have $\widehat{\g{osp}}_\mathrm{res}(\mathscr K)=\g{osp}_\mathrm{res}(\mathscr K)\oplus\R$ as a vector space, with the superbracket
\[
[(T,z),(T',z')]=([T,T'],\omega(T,T')).
\]

We now describe the Fock space realization of the metaplectic representation of $\widehat{\g{osp}}_\mathrm{res}(\mathscr K)$.
We choose an orthonormal basis $\{f_1,f_2,f_3,\ldots\}$ for the fermionic space $\mathscr K_\eev$ and 
an orthonormal basis 
$\{b_1,b_2,b_3,\ldots\}$ for the bosonic space $\mathscr K_\ood$.
For every two integers $k,l\geq 0$, we define 
$\mathscr F^{k,l}$ to be the complex vector space spanned by monomials
\begin{equation}
\label{standardm}
f_1^{r_1}f_2^{r_2}f_3^{r_3}\cdots b_1^{s_1}b_2^{s_2}b_3^{s_3}\cdots
\end{equation}
with the following properties.
\begin{itemize}
\item[(i)] For every positive integer $m$, we have $r_m\in\{0,1\}$ and $s_m\in\{0,1,2,3,\ldots\}$.
\item[(ii)] For all but finitely many $m$, we have $r_m=s_m=0$.
\item[(iii)] $\sum_{m=1}^\infty r_m=k$ and $\sum_{m=1}^\infty s_m=l$.
\end{itemize}
We will refer to the monomials satisfying the above properties as \emph{reduced} monomials.
To simplify the notation, we will also use more general monomials of the form 
$
v_1\cdots v_m
$
where $v_k\in\mathscr K_\eev\cup\mathscr K_\ood$ for every $1\leq k\leq n$. Observe that any such monomial can be expressed as a linear combination of reduced monomials  using linearity and the relations 
\[
f_mf_n=-f_nf_m\, \text{ and }\, b_mb_n=b_nb_m
\]
for every two nonnegative integers $m,n$.

We set $\mathscr F=\bigoplus_{k,l\geq 0}\mathscr F^{k,l}$ and define an inner product 
$\langle\cdot,\cdot\rangle_\mathscr F$ on $\mathscr F$
as follows. If $v=f_1^{r_1}f_2^{r_2}\cdots b_1^{s_1}b_2^{s_2}\cdots\in\mathscr F^{k,l}$ and 
$w=f_1^{r'_1}f_2^{r'_2}\cdots b_1^{s'_1}b_2^{s'_2}\cdots\in\mathscr F^{k',l'}$ then 
\[
\langle v,w\rangle_\mathscr F
=
\begin{cases}
1 &\text{ if }r_k=r_k' \text{ and } s_k=s_k'\text{ for every }k\geq 1,\\
0 &\text{ otherwise.}
\end{cases}
\]
Next we describe the action of $\widehat{\g{osp}}_\mathrm{res}(\mathscr K)$ on $\mathscr F$. Let $(T,z)\in\widehat{\g{osp}}_\mathrm{res}(\mathscr K)$
where $T$ is 
a homogeneous element
expressed in the form
$T=T_\mathrm{lin}+T_\mathrm{conj}$. For every $v\in\mathscr F$ we set
\[
\rho^\mathscr F\big((T,z)\big)v=\rho^\mathscr F(T_\mathrm{lin})v+\rho^\mathscr F(T_\mathrm{conj})v+
\sqrt{-1}z\cdot v
\]  
where
$\rho^\mathscr F(T_\mathrm{lin})$ and $\rho^\mathscr F(T_\mathrm{conj})$ 
are defined as follows. If $v=v_1\cdots v_m$ is a monomial then
\[
\rho^\mathscr F(T_\mathrm{lin})v=\sum_{r=1}^m(-1)^{p(T_\mathrm{lin})\big(p(v_1)+\cdots+p(v_{r-1})\big)}
v_1\cdots v_{r-1}(T_\mathrm{lin}v_r)v_{r+1}\cdots v_{m}.
\]
We also define $\rho^\mathscr F(T_\mathrm{conj})v$ by
\[
\rho^\mathscr F(T_\mathrm{conj})v=
a(T_\mathrm{conj})v-a(T_\mathrm{conj})^\dag v
\]
where $a(T_\mathrm{conj}):\mathscr F\to\mathscr F$ and $a(T_\mathrm{conj})^\dag:\mathscr F\to\mathscr F$ 
are linear maps defined as follows. If $v\in\mathscr F^{k,l}$ then
\[
a(T_\mathrm{conj})v=\lambda_{k,l}
\Big(
\sqrt{-1}\sum_{r=1}^\infty (T_\mathrm{proj}b_r)b_r+\sum_{r=1}^\infty (T_\mathrm{proj}f_r)f_r
\Big)v
\]
where $\lambda_{k,l}=\frac{1}{2}\sqrt{(k+l+1)(k+l+2)}$.
Moreover, $a(T_\mathrm{conj})^\dag$ is the superadjoint of $a(T_\mathrm{conj})$ on $\mathscr F$, i.e.,
\[
a(T_\mathrm{conj})^\dag=
\begin{cases}
a(T_\mathrm{conj})^*&\text{ if }T\text{ is even}\\
-\sqrt{-1}a(T_\mathrm{conj})^*&\text{ if }T\text{ is odd}\\
\end{cases}
\]
where $a(T_\mathrm{conj})^*$ is the adjoint of $a(T_\mathrm{conj})$ on $\mathscr F$, which is defined by 
\[
\langle\, a(T_\mathrm{conj})^*w\,,\,w'\,\rangle_\mathscr F
=
\langle\, w\,,\,a(T_\mathrm{conj})w'\,\rangle_\mathscr F\,\text{ for every }w,w'\in\mathscr F.
\]
The restriction of the above action to $\widehat{\g{osp}}_\mathrm{res}(\mathscr K)_\eev$ is the tensor product of the spin representation of 
$\widehat{\g{o}}_\mathrm{res}(\mathscr K_\eev)$ and the metaplectic representation of $\widehat{\g{sp}}_\mathrm{res}(\mathscr K_\ood)$.
This representation of $\widehat{\g{osp}}_\mathrm{res}(\mathscr K)_\eev$ integrates to an
analytic unitary representation $(\pi,\mathscr H)$ of $\widehat{\mathrm{OSp}}_\mathrm{res}(\mathscr K)$ on the completion of $\mathscr F$.  
For every $(T,z)\in\widehat{\g{osp}}_\mathrm{res}(\mathscr K)$, the space $\mathscr F$ consists of analytic vectors for the operator $\rho^\mathscr F\big((T,z)\big)$ 
\cite[Sec. 9]{neebconflu}.  From Lemma \ref{lem-criteriaforesa} 
it follows that $(\pi,\mathscr H,\rho^\mathscr F,\mathscr F)$
is a pre-representation of 
$(\widehat{\mathrm{OSp}}_\mathrm{res}(\mathscr K),\widehat{\g{osp}}_\mathrm{res}(\mathscr K))$. 
Consequently, Theorem \ref{thm-stabili} implies the following result.
\begin{theorem}
Let $\overline{\mathscr F}$ be the Hilbert space completion of the Fock space $\mathscr F$ defined above. Then
there exists a unique analytic unitary representation 
$(\sigma,\rho^\sigma,\overline{\mathscr F})$
of 
$(\widehat{\mathrm{OSp}}_\mathrm{res}(\mathscr K),\widehat{\g{osp}}_\mathrm{res}(\mathscr K))$ 
with the following properties.
\begin{enumerate}
\item[(i)] $\mathscr F\sseq \overline{\mathscr F}^\omega$, i.e., every $v\in\mathscr F$ is an analytic vector 
for $(\sigma,\overline{\mathscr F})$.
\item[(ii)] $\rho^\sigma(x)\big|_\mathscr F=\rho^\mathscr F(x)$ for every 
$x\in\widehat{\g{osp}}_\mathrm{res}(\mathscr K)$.
\end{enumerate}
\end{theorem}

\appendix

\section{A smooth non-analytic unitary representation}

\label{exampleanalytic}

The goal of this appendix is to give two examples: a smooth unitary representation of a Banach--Lie group without nonzero analytic vectors, and an analytic unitary representation of a Banach--Lie group without nonzero bounded vectors.

We start with the first example. In this example the Hilbert space of the representation is 
$\cH = L^2([0,1],\C)$ and $G$ is the additive group 
of a Banach space $\g g$ of measurable functions $[0,1] \to \R$ 
with the property that 
\[ L^\infty([0,1],\R)\subeq \g g\subeq \bigcap_{p \in \N} 
L^p([0,1],\R).  \] 
Using results  from \cite{khdifferentiable}, these two inclusions 
easily imply that the representation $(\pi,\mathscr H)$
of $G$  given by 
$\big(\pi(f)\xi\big)(x) = e^{if(x)}\xi(x)$  is smooth. 
Now the main point is to choose $\g g$ large enough so that the 
space $\cH^\omega$  does not contain nonzero vectors. 

Let  $g(x)=e^{-\sqrt{x}}$ for $x\geq 0$. Then $\int_0^\infty x^ng(x)dx=2\cdot\Gamma(2n+2)<\infty$ 
for every $n\in\N$ while $\int_0^\infty e^{tx} g(x)dx=\infty$ for every $t>0$. 
Consider the map
\[
G:[0,\infty)\to[0,\infty)\ ,\ G(x)=\frac{1}{2}\int_0^x g(t)dt.
\]
As $g$ is continuous, the function $G$ is $C^1$ with 
$G'(x) > 0$ for every $x \geq 0$. Next observe that 
\[ \lim_{x \to \infty} G(x) = 
\frac{1}{2} \int_0^\infty e^{-\sqrt t}\, dt 
= \int_0^\infty e^{-x}x\, dx 
= \Gamma(2) = 1.\] 
Therefore $G\: [0,\infty) \to [0,1)$ is a $C^1$-diffeomorphism. 
Set
\begin{equation}
  \label{eq:h}
h(x) = 
\begin{cases} 
G^{-1}(1-x)&\text{if  $0 < x \leq 1$},\\ 
0&\text{if $x=0$.}
\end{cases}
\end{equation}
Then $h : [0,1] \to [0,\infty)$ is a Lebesgue measurable function 
with a singularity at~$0$.

In the following we denote the Lebesgue measure of a 
measurable set $E\sseq [0,1]$ by $\mu(E)$. 
We say that the {\it metric density of $E$ at $x_0$ is $1$} 
if $\lim_{\eps \to 0} \frac{\mu(E \cap [x_0-\eps, x_0+\eps])}{2\eps} = 1$. 
According to \cite[\S 7.12]{Ru86} at almost every point of $E$ 
the metric density of $E$ is $1$. Clearly, $0$ and $1$ can never 
have this property for $E \subeq [0,1]$. Note that
at every point $x_0$ of metric density $1$, we also have 
\[ \lim_{\eps \to 0} \frac{\mu(E \cap [x_0, x_0+\eps])}{\eps} = 1.\]

\begin{lemma}
\label{seriesfub}
If $\{a_n\}_{n=0}^\infty$ and $\{s_n\}_{n=0}^\infty$ are 
sequences of non-negative real numbers such that $\{a_n\}_{n=0}^\infty$ is decreasing and $a_n \to 0$, then 
\[ \sum_{n = 0}^\infty (a_n - a_{n+1}) (s_0 + \cdots + s_n)
=  \sum_{n = 0}^\infty a_n s_n.\]    
\end{lemma}
\begin{proof}
For every two non-negative integers $p,q$ set 
\[
b_{p,q}=
\begin{cases}
(a_p-a_{p+1})s_q&\text{if }p\geq q,\\
0&\text{otherwise.}
\end{cases}
\]
 By \cite[Cor. 1.27]{Ru86} 
we have
$
\sum_{p=0}^\infty\sum_{q=0}^\infty b_{p,q}=\sum_{q=0}^\infty\sum_{p=0}^\infty b_{p,q}
$. 
The lemma follows easily from the latter equality.
\end{proof}

\begin{lemma} \label{lem:lebes} 
Let $H : (0,1] \to \R$ be a continuous and decreasing map
such that $\int_0^1 H(x)\, dx = \infty$ and 
$E \subeq [0,1]$ be a measurable set 
such that $\lim_{\eps \to 0} \frac{\mu(E \cap [0,\eps])}{\eps} = 1$.
Then $\int_E e^{tH(x)}\, dx = \infty$ for every $t > 0$. 
\end{lemma}

\begin{proof} Our assumption implies that
$\lim_{x \to 0} H(x) = \infty$ because otherwise $H$ would be bounded, 
hence integrable. Adding a constant to $H$ will not affect the statement of the lemma. Therefore 
we may assume 
that $H(1) = 1$. 
We now put $\eps_n := H^{-1}(2^n)$ and note that 
$\eps_0 = 1$ as well as $\eps_n \to 0$. 
We now find that 
\[ \infty = \int_0^1 H(x)\, dx 
= \sum_{n = 0}^\infty \int_{\eps_{n+1}}^{\eps_n} H(x)\, dx 
\leq \sum_{n = 0}^\infty 2^{n+1} (\eps_n - \eps_{n+1}).\] 
Lemma \ref{seriesfub} implies that  
\begin{equation}
  \label{eq:inf0} 
\sum_{k = 0}^\infty \eps_k 2^k 
= \sum_{n = 0}^\infty (2^{n+1}-1) (\eps_n - \eps_{n+1}) = \infty 
\end{equation}
because $\sum_{n = 0}^\infty \eps_n - \eps_{n+1} 
=  \eps_0$. 

If $E_n = E \cap [0,\eps_n]$ for every $n\geq 0$, then using Lemma \ref{seriesfub} we have
\begin{align*}
\int_{E} H(x)\, dx 
&= \sum_{n=0}^\infty \int_{E_n\setminus E_{n+1}} H(x)\, dx 
\geq \sum_{n=0}^\infty 2^n(\mu(E_n) - \mu(E_{n+1}))\\
&=  \sum_{n=0}^\infty (1 + 1 + 2^1 + \cdots + 2^{n-1})(\mu(E_n) - \mu(E_{n+1}))\\
&=  \mu(E_0) + \sum_{n=1}^\infty 2^{n-1} \mu(E_n).
\end{align*}
Since
$\lim_{n \to \infty} \frac{\mu(E_n)}{\eps_n} =1$, from 
\eqref{eq:inf0} it follows that  $\int_E H(x)\, dx = \infty$. 
\end{proof}

\begin{lemma} \label{lem:1.4} 
The map $h$ given in \eqref{eq:h} has the following properties: 
\begin{enumerate}
\item[(i)] $h \in \bigcap_{p \in \N} L^p([0,1],\R)$.    
\item[(ii)] $h\big|_{(0,1]}$ is strictly decreasing with 
$\lim_{x \to 0} h(x) = \infty$. 
\item[(iii)] If $E \subeq [0,1]$ is a measurable subset 
satisfying $\lim_{\eps \to 0} \frac{\mu(E \cap [0,\eps])}{\eps} = 1$, 
then $\int_E e^{th(x)}\, dx = \infty$ for every $t > 0$. 
\end{enumerate}
\end{lemma}

\begin{proof}
(i) For every $n\in\N$ we have
\begin{align*}
\int_0^1 h(x)^n\, dx = \int_0^1 h(1-x)^n\, dx 
&=  \int_0^\infty h(1-G(y))^n|G'(y)|\, dy\\
&=  \frac{1}{2} \int_0^\infty y^n g(y)\, dy < \infty.
\end{align*}

(ii) Follows from the definition of $h$.

(iii) For every $t > 0$ we have
\begin{align*}
\int_0^1 e^{t h(x)} \, dx = \int_0^1 e^{t h(1-x)} \, dx 
&=  \int_0^\infty e^{th(1-G(y))}|G'(y)|\, dy\\
&=  \frac{1}{2} \int_0^\infty e^{ty}g(y)\, dy = \infty.
\end{align*}
Lemma \ref{lem:lebes} completes the argument.
\qedhere
\end{proof}

Let $\|\cdot\|_p$ denote the usual norm of $L^p([0,1],\R)$.
Set $c_n = \|h\|_n$ for every $n \in \N$. Note that $c_n > 0$ for 
every $n \in \N$. Since $h$ is unbounded, for every
 $c > 0$ the set 
\[
I_c = \{ x \in [0,1]\, :\, |h(x)| \geq c\}
\]
has positive measure. 
This implies that 
$\|h\|_n \geq \sqrt[n]{\mu(I_c)} c$. Since the right hand side converges 
to $c$ when $n\to\infty$, it follows that  $\lim_{n\to\infty} \|h\|_n = \infty$. 

We are now ready to define the Banach space $\g g$.
For every
measurable function $f : [0,1] \to \R$  we define a norm
\[ \|f\| = \sup \Big\{ \frac{\|f\|_n}{c_n} \,:\, n \in \N\Big\}, \]
and set \[ 
\g g= \Big\{ f \in \bigcap_{p \in \N} L^p([0,1],\R) \,:\, 
\|f\| < \infty \Big\}. \] 
It is fairly straightforward to check that $\g g$ is a Banach space and 
$L^\infty([0,1],\R) \subeq \g g$. We set $G=\g g$, i.e., the additive group of the Banach space $\g g$.

By construction, $h \in \g g$. Next we observe that 
we may also identify $\g g$ with a space of $1$-periodic functions on 
$\R$. Then the norm defined on $\g g$ is translation invariant. 
Therefore $\g g$ also contains the functions 
$h_{x_0}(\cdot) = \tilde h(\cdot - x_0)\big|_{[0,1]}$, where $\tilde h$ is the 
$1$-periodic extension of $h$ to $\R$.  
For $x_0 < 1$ and $0 < \eps < 1 - x_0$, it satisfies 
\begin{equation}\label{eq:inf2} 
\int_{x_0}^{x_0+ \eps} e^{t h_{x_0}(x)}\, dx = \infty 
\quad \mbox{ for } \quad t > 0
\end{equation} 
by Lemma~\ref{lem:1.4}(iii).

\begin{theorem} \label{thm:1.9} Let $\cH = L^2([0,1],\C)$, $G=\g g$ be as above, and 
$(\pi,\mathscr H)$ be the unitary representation of
$G$ defined by 
$\big(\pi(f)\xi\big)(x) = e^{if(x)} \xi(x)$. Then $(\pi,\mathscr H)$ 
is smooth and $\cH^\omega = \{0\}$. 
\end{theorem}

\begin{proof} According to \cite[Sec. 10]{khdifferentiable} an 
element $\xi \in \cH$ is a smooth vector if and only if 
\[  \|f^n \xi\|_2 < \infty \quad  \mbox{ for every } n \in \N\text{ and every } f \in \g g.\] 
Hence the inclusion 
$\g g\subeq \bigcap_{p \in \N} L^p([0,1],\R)$ implies that all 
bounded functions are smooth vectors. In particular, 
$(\pi, \cH)$ is a smooth representation. 

By Lemma \ref{convoncplx} and Theorem \ref{firstthmbosi} an element $\xi \in \cH$ is analytic if and only if 
$\sum_{n=0}^\infty \frac{\|f^n \xi\|_2}{n!}$ 
converges on some neighborhood of the origin in $\g g$. 
If $\xi$ is non-zero, then there exists an 
$\eps > 0$ for which the subset 
$E = \{\,x\in[0,1]\,:\, \eps < |\xi(x)| < 1/\eps\,\}$ has positive measure. 
Let $\chi_E$ denote the characteristic function of the set $E$.
Analyticity of $\xi$ leads to the  estimates 
\[\int_E e^{|f(x)|}\, dx =  \sum_{n=0}^\infty \frac{\int_E |f(x)|^n\, dx}{n!}  
\leq \sum_{n=0}^\infty \frac{\|f^n\chi_E\|_2}{n!}  
\leq \frac{1}{\eps}\sum_{n=0}^\infty \frac{\|f^n\xi\|_2}{n!}<\infty  \] 
for all elements $f$ in a neighborhood of the origin in $\g g$.

Let $x_0 \in (0,1)$ be a point of metric density $1$ of $E$ 
and recall that this implies that the sets 
$E_\eps := E \cap [x_0, x_0 + \eps]$ satisfy 
$\lim_{\eps \to 0} \frac{\mu(E_\eps)}{\eps} = 1$. 
Therefore Lemma~\ref{lem:lebes} implies that 
the function $f(x) = \tilde h(x - x_0)$ satisfies 
$\int_E e^{tf(x)}\, dx = \infty$ for every $t~>~0$. 
This contradiction shows that there is no non-zero analytic vector 
in~$\cH$. 
\end{proof}

\section{An analytic representation without  bounded vectors}

\label{examplebounded}

In this appendix we give an example of an analytic unitary representation of a Banach--Lie group without nonzero bounded vectors. The notation in this appendix is the same as in Appendix \ref{exampleanalytic}.

\begin{definition} Let $G$ be a Banach--Lie group and $(\pi,\mathscr H)$ be a unitary representation of $G$. 
We call  $(\pi, \cH)$ 
\emph{bounded} if 
$\pi : G \to \mathrm U(\cH)$ is continuous with respect to the operator 
norm on $\mathrm U(\cH)$. 
A vector 
$v \in \cH$ is said to be {\it bounded} if the representation 
of $G$ on the closed invariant subspace $\cH_v = \oline{\mathrm{Span}(\pi(G)v)}$ 
is bounded. The subspace of bounded vectors in $\mathscr H$  is denoted by $\mathscr H^b$. 
The representation $(\pi, \cH)$ is said to be {\it locally bounded} 
if $\cH^b$ is dense in $\cH$. 
\end{definition}
Since $G$ and $\mathrm U(\cH)$ are Banach--Lie groups, 
every bounded representation is in particular analytic as a map 
$G \to \mathrm U(\cH)$. In particular, $\cH^b \subeq \cH^\omega$. 

Observe that a representation is locally bounded if and only if it is a direct 
sum of bounded representations. In fact, if $\cH^b$ is dense, then 
a standard application of Zorn's Lemma shows that 
$\cH$ is an orthogonal direct sum of cyclic subspaces generated by 
bounded vectors.

For an element $x$ of a vector space $V$, we write $x^*(\alpha) = \alpha(x)$ 
for the corresponding linear functional on the dual space $V^*$.

\begin{rmmk}
\label{remk-abc} 
(i) Let $G = (V,+)$ be the additive group of a finite 
dimensional real vector space and $\mu$ be the Lebesgue measure 
on $V^*$. Then $\pi(x)f = e^{ix^*}f$ 
defines a continuous unitary representation 
on $\cH = L^2(V^*,\mu)$. A vector 
$f \in \cH$ is bounded if and only if 
$f$ vanishes almost everywhere outside some compact subset. 
Clearly this condition is stronger than $f \in \cH^\omega$, which is 
equivalent to $e^{x^*} f$ 
being square integrable for every $x$ in a neighborhood of~$0$ in~$V$.

(ii) When $G = (V,+)$ is the additive group of a finite dimensional vector space, Bochner's 
Theorem asserts that every continuous positive definite 
function $\phi : G \to \C$ is the Fourier transform 
\[ \phi = \hat\mu, \quad \hat\mu(x) = \int_{V'} e^{i\alpha(x)}\, d\mu(\alpha)\] 
of a finite positive regular Borel measure $\mu$ 
on the dual space $V^*$. Then 
the representation of $\cH = L^2(V^*,\mu)$ by 
$\pi(x)f = e^{ix^*}f$ is cyclic, 
generated by the constant function $1$, and 
$\la \pi(x)1,1\ra = \hat\mu(x) = \phi(x)$. The description of 
the bounded and analytic vectors under (a) remains the same in this 
situation. 

(iii) If $G = (V,+)$ is the additive group of a Banach space 
and $\mu$ is a regular positive Borel measure on the topological dual space 
$V'$ with respect to the 
weak-$*$-topology, then we also obtain a unitary 
representation of $G$ on $\cH := L^2(V',\mu)$ by 
$\pi(x)f = e^{ix^*}f$. 

Every weak-$*$-compact subset of $V'$ is weakly bounded, hence bounded 
by the Banach--Steinhaus Theorem. Therefore the compact subsets of 
$V'$ are precisely the weak-$*$-closed bounded subsets. All 
closed balls in $V'$ have this property. 
If $\mu$ is supported by a bounded set, then one easily verifies 
that the representation $\pi$ is bounded. 
If this is not the case, then 
$f \in \cH^b$ is equivalent to 
$f$ vanishing $\mu$-almost everywhere outside a sufficiently 
large ball. Since $\mu$ is regular, this implies that 
$\pi$ is locally bounded. 

\end{rmmk}

\begin{theorem} Let $G=(V,+)$ be the additive group of a Banach space. For a positive definite function 
$\phi$ on  $G$ the corresponding 
GNS representation $(\pi_\phi, \cH_\phi)$ is locally bounded if 
and only if there exists a regular positive Borel measure 
$\mu$ on $V'$ with $\phi = \hat\mu$. 
\end{theorem}

\begin{proof} If $\phi = \hat\mu$ for a regular positive Borel 
measure on $V'$, then 
the GNS representation defined by $\phi$ is isomorphic 
to the cyclic subrepresentation of $L^2(V',\mu)$ generated 
by the constant function $1$, hence locally bounded 
by Example~\ref{remk-abc}(iii). 

Conversely, suppose that $(\pi_\phi, \cH_\phi)$ is locally 
bounded, i.e., it can be expressed as a direct sum 
$\widehat\bigoplus_{j \in J} (\pi_j,\cH_j)$
of bounded representations. 
Writing $\phi = \sum_{j \in J} v_j$ with $v_j \in \cH_j$, 
the orthogonality of the family $\{v_j\, : \,{j\in J}\}$ implies that 
only countably many of them are non-zero, and since 
$\phi$ is cyclic in $\cH_\phi$, the index set $J$ is countable. 
Suppose that all the functions 
$\phi_j(g) = \la \pi_j(g)v_j, v_j\ra$ are Fourier transforms 
of positive regular Borel measures $\mu_j$ on $V'$. Then 
$\phi = \sum_j \phi_j = \sum_j \hat\mu_j$ is the Fourier 
transform of the positive Borel measure $\mu = \sum_{j \in J} \mu_j$. 
Therefore without loss of generality we may  assume that $\pi_\phi$ is a bounded 
representation. Then the Spectral Theorem for 
semibounded representations \cite[Thm.~4.1]{Ne09} 
implies the existence of a regular Borel spectral measure 
$P$ on some weak-$*$-compact subset $X \subeq V'$ with 
$\pi_\phi(x) = \int_X e^{i\alpha(x)}\, dP(\alpha)$. 
For the cyclic vector $v \in\cH_\phi$ satisfying 
$\phi(x) = \la \pi_\phi(x)v, v \ra$, this leads 
\[ \phi(x) = \la \pi_\phi(x)v, v \ra 
= \int_X e^{i\alpha(x)}\, dP^v(\alpha) = \widehat{P^v}(x),\] 
where $P^v(E) = \la P(E)v,v\ra$ is the positive regular Borel 
measure associated to $v$ and $P$.
\end{proof}

The preceding discussion shows that the locally bounded cyclic 
representations are precisely those that can be realized in spaces 
$L^2(V',\mu)$ for regular Borel measures on $V'$. 
For a  representation $(\pi, \cH)$ with no non-zero bounded vector, 
for no non-zero $v \in \cH$, the positive definite function 
$\pi^{v,v}(x) = \la \pi(x)v,v\ra$ is a Fourier transform 
of a regular Borel measure on $V'$. In this sense they are very singular. 
In the light of this discussion, it is a natural question 
how big the difference between analytic and bounded representations 
really is. 

From now on, for every measurable function $f$ we set 
\[
\|f\|= \sup \Big\{\, \frac{\|f\|_n}{\sqrt[n]{n!}}\, :\, n \in \N\,\Big\}.
\]
The following theorem shows that analytic representations of 
Banach--Lie groups need not contain non-zero bounded vectors.

\begin{theorem} The space 
\[ \g g= \Big\{ f \in \bigcap_{p \in \N} L^p([0,1],\R) \,:\, 
\|f\|  < \infty \Big\} \] 
is a Banach space. The unitary representation 
$(\pi, L^2([0,1],\C))$ of $G = (\g g,+)$ given by 
$\big(\pi(f)\xi\big)(x) = e^{if(x)}\xi(x)$ is analytic with $\cH^b = \{0\}$. 
\end{theorem}

\begin{proof} To prove that
$\g g$ is indeed a Banach space is straightforward and we leave it to the reader. 
As $\lim_{n \to \infty} \sqrt[n]{n!} = \infty$, we also have 
$L^\infty([0,1],\R) \subeq \g g$, so that the constant function 
$1$ is a cyclic vector. To show that $(\pi, \cH)$ is analytic, 
it therefore suffices to show that $1 \in \cH^\omega$. 

We claim that the series 
$\sum_{n = 0}^\infty \frac{1}{n!} \|\dd\pi(f)^n1\|_2
= \sum_{n = 0}^\infty \frac{1}{n!} \|f^n\|_2$ 
converges uniformly for $\|f\| < \frac{1}{2}$, and this implies 
the analyticity of $1$. Below we need the estimate 
\[ (2n)! 
= (1 \cdots 3 \cdots (2n-1))(2 \cdots 4 \cdots 2n)
\leq (2 \cdots 4 \cdots 2n)^2 
= 2^{2n}(n!)^2.\] 
This leads to 
\begin{align*} 
\sum_{n = 0}^\infty \frac{1}{n!} \|f^n\|_2 
= \sum_{n = 0}^\infty \frac{1}{n!} \|f\|_{2n}^n 
\leq \sum_{n = 0}^\infty \frac{\sqrt{(2n)!}}{n!} \|f\|^n
\leq \sum_{n = 0}^\infty 2^n \|f\|^n = \frac{1}{1 - 2\|f\|}.
\end{align*}

Next we show that $\cH^b = \{0\}$. Suppose that 
$\xi$ is a non-zero bounded vector, i.e., that the representation of 
$G$ on the cyclic subspace $\cH_\xi$ generated by $\xi$ is bounded. 
This implies in particular that $\cH_\xi$ is invariant under the 
derived representation, i.e., under multiplication with elements of 
$\g g$. As $\xi$ is non-zero, there exists an $\eps > 0$ for which the set 
$E = \{ x \in [0,1] : |\xi(x)| \geq \eps\}$ has positive measure. 
Since $(\xi\big|_E)^{-1}$ is bounded, 
the characteristic function $\chi_E$ of $E$ is contained in $\cH_\xi$, 
and further $L^\infty([0,1],\R)\cdot \chi_E \subeq \cH_\xi$ implies that 
$L^2(E) \subeq \cH_\xi$. 

The boundedness of the representation of $G$ on $L^2(E)$ implies in 
particular for each $f \in \g g$  that 
$\|e^{tf\big|_E} - \chi_E\|_\infty \to 0$ for $t \to 0$, and hence 
that $f\big|_E$ is essentially bounded. 
For $f(x) = \log(x)$ we have 
\[ \|f\|_n^n 
= \int_0^1 |\log(x)|^n\, dx 
= \int_0^\infty y^n e^{-y}\, dy = \Gamma(n+1) = n!, \] 
which shows that $\log(\cdot ) \in \g g$. 
Let $\widetilde\log$ denote the $1$-periodic extension of 
$\log$ to $\R$ and 
$\log_{x_0}(\cdot) = \widetilde\log(\cdot - x_0)\big|_{[0,1]}$. 
Then the translation invariance of the norm defining $\g g$ implies 
that $\log_{x_0} \in \g g$. 

As $E$ has positive measure, there exists a point $x_0 \in (0,1)\cap E$ 
of metric density $1$, hence an $\eps_0 > 0$ for which the set
$E_\eps = \{ x \in E \,:\, x_0 \leq x \leq x_0 + \eps\}$ 
has positive measure for every 
$\eps \in (0,\eps_0)$. This implies 
that $\log_{x_0}$ is not essentially bounded on $E$. 
From this contradiction we derive that $\cH^b =\{0\}$. 
\end{proof}

\section{Analytic functions in the Banach context}
\label{section-appendix-analyticmaps}

Let $\mathscr A$ and $\mathscr B$ be two Banach spaces over $\K$, where $\K\in\{\R,\C\}$. A \emph{homogeneous polynomial} 
of degree $n$ from $\mathscr A$ to $\mathscr B$ is a map 
\[
p:\mathscr A\to\mathscr B\;,\; p(v)=F(v,\ldots,v)
\] where
$
F:\mathscr A\times\cdots\times \mathscr A\to\mathscr B
$ is a symmetric $\mathbb K$-multilinear map.  The homogeneous polynomial $p$ is continuous if and only if $F$ is continuous.

If $U\sseq \mathscr A$ is an open set, then
a continuous function $f:U\to \mathscr B$ is called \emph{analytic} in $U$ if and only if for every $u\in U$ 
there exist a neighborhood $V_u$ of the origin in $\mathscr A$ and continuous homogeneous polynomials 
$p_n:\mathscr A\to \mathscr B$   such that $\deg(p_n)=n$, $u+V_u\sseq U$, and 
\begin{equation}
\label{eq-f(u+v)}
f(u+v)=\sum_{n=0}^\infty p_n(v)\ \ \text{for every $v\in V_u$}.
\end{equation}
\begin{rmk}
The convergence of the series \eqref{eq-f(u+v)} is pointwise. 
However, 
Theorem \ref{firstthmbosi} below implies that if the series \eqref{eq-f(u+v)} converges in $U$ pointwise, 
then for every  $u\in U$ the series also converges normally  at $u$  
(i.e., absolutely uniformly in a neighborhood of $u$). 
In some references, e.g., \cite[No. 3.2]{bourbakivariete}, analytic maps are defined on the basis 
of the latter form of convergence.  
\end{rmk}

Recall that a subset $S$ of a vector space $\mathscr V$ over $\K$ is called \emph{absorbing} if for every 
$v\in\mathscr V$ there exists a $t_v>0$ such that for all $c\in \K$, if $|c|\leq t_v$ then
$c\cdot v\in S$.

\begin{theorem}
\label{firstthmbosi}
Let $\mathscr A$ and $\mathscr B$ be Banach spaces over $\mathbb K$ where $\mathbb K\in\{\R,\C\}$.
Let $S\sseq \mathscr A$ be an absorbing set. For every integer $n\geq 0$, let $p_n:\mathscr A\to \mathscr B$ be a continuous 
homogeneous polynomial of degree $n$. Consider the formal series
\[
\phi(v)=\sum_{n=0}^\infty p_n(v).
\]
Then the following statements hold.
\begin{itemize}
\item[(i)]
Suppose that there exists an absorbing set $S\sseq \mathscr A$ such that 
the series $\phi(v)$ converges for every $v\in S$. Then  
there exists an open neighborhood $U$ of the origin in $\mathscr A$ such that
\[
\sum_{n=0}^\infty\sup\{\|p_n(v)\|\ :\ v\in U\}<\infty.
\]
\item[(ii)] Suppose that there exists an open neighborhood $V$ of the origin in $\mathscr A$ such that the series $\phi(v)$ converges for 
every $v\in V$. Then the function \[
\phi\big|_V:V\to \mathscr B
\] is analytic in $V$.

\end{itemize}
\end{theorem}

\begin{proof}
Statement (i) follows from \cite[Prop. 5.2]{bochnaksiciak} and \cite[Thm. 5.2]{bochnaksiciak}.
Statement (ii) follows from \cite[Thm. 5.2]{bochnaksiciak}
\end{proof}
\begin{rmk}
When
$\mathscr A=\C^n$ and $\mathscr B=\C$,  Theorem \ref{firstthmbosi} implies a result which is originally due to Hartogs. 
A proof of this special case is given in  \cite[Thm. 1.5.6]{rudinfncth}.

\end{rmk}

Let $\mathscr A$ be a real Banach space and $\mathscr B$ be a complex Banach space (which can also be considered 
as a real Banach space). Every $\R$-multilinear map
$F:\mathscr A\times \cdots \times \mathscr A\to \mathscr B$ can be extended to a $\C$-multilinear map
$F^\mathbb C:\mathscr A^\C\times \cdots\times \mathscr A^\C\to \mathscr B$
by extension of scalars. Therefore every continuous homogeneous polynomial $p:\mathscr A\to\mathscr B$ extends to a continuous homogeneous
polynomial $p^\C:\mathscr A^\C\to \mathscr B$.

\begin{theorem}
\label{secondthmbosi}
Let $\mathscr A$ be a real Banach space and $\mathscr B$ be a complex Banach space. For every integer $n\geq 0$, 
let $p_n:\mathscr A\to \mathscr B$ be a  
continuous homogeneous polynomial of degree $n$. Suppose that the formal series
\[
\phi(v)=\sum_{n=0}^\infty p_n(v)
\]
converges for every $v\in U$, where $U$ is an open neighborhood of zero in $\mathscr A$.
Then there exists an open neighborhood $U^\C$ of zero in $\mathscr A^\C$ such that 
$U\sseq U^\C$, the series 
\[
\phi^\C(v)=\sum_{n=0}^\infty p^\C_n(v)
\]
converges for every $v\in U^\C$, and the map $
\phi^\C\big|_{U^\C}:U^\C\to \mathscr B$
 is analytic in $U^\C$.
\end{theorem}
\begin{proof}
The statements of the theorem follow from \cite[Prop. 5.4]{bochnaksiciak} and Theorem \ref{firstthmbosi}(ii) above.
\end{proof}
An analogue of Hartogs' Theorem is also valid in this framework.
\begin{theorem}
\label{thirdthmbosi}
Let $\mathscr A$, $\mathscr B$, and $\mathscr C$ be complex Banach spaces and $U\sseq \mathscr A\times \mathscr B$ be an 
open set. If a function $f:U\to \mathscr C$ is separately analytic, then it is analytic.
\end{theorem}
\begin{proof}
This is \cite[Cor. 6.2]{bochnaksiciak}.
\end{proof}
Recall that in a complex topological vector space, a neighborhood $W$ of the origin is called \emph{balanced} if and only if for every
$z\in \C$ such that $|z|\leq 1$ and every $w\in W$ we have $z\cdot w\in W$.
\begin{theorem}
\label{fourththmbosi}
Let $\mathscr A$ and $\mathscr B$ be complex Banach spaces, $U\sseq \mathscr A$ be an open set, and $f:U\to \mathscr B$ be 
analytic in $U$. 
Let $u\in U$ and $W$ be a balanced open neighborhood of zero in $\mathscr A$ such that $u+ W\sseq U$.
For every integer $n\geq 0$, set \[
\delta_u^{(n)}f(v)=\frac{d^n}{d\zeta^n}f(u+\zeta \cdot v)\big|_{\zeta=0}.
\] Then the following statements hold.
\begin{itemize}
\item[(i)] For every $n$, $\delta_u^{(n)}f$ is a continuous homogeneous polynomial  of degree $n$. 
\item[(ii)] $
\displaystyle f(u+v)=\sum_{n=0}^\infty \frac{1}{n!}\delta_u^{(n)}f(v)\ \text{for every }v\in W$.
\end{itemize}
\end{theorem}
\begin{proof}
The above statements are consequences of \cite[Prop. 5.5]{bochnaksiciak}.
\end{proof}


\begin{thebibliography}{aaaaaaa}




\bibitem[AlHiLa]{allhilglaub} Alldridge, A., Hilgert, J., Laubinger, J., \emph{Harmonic analysis on Heisenberg--Clifford Lie supergroups}, preprint, \textsf{arXiv:1102.4475}.

\bibitem[AlLa]{alldridge} Alldridge, A., Laubinger, M.,  \emph{Infinite-dimensional supermanifolds over arbitrary base fields}, Forum Mathematicum, to appear.


\bibitem[BoSi1]{bochnaksiciak1}
Bochnak, J., Siciak, J.,  \emph{Polynomials and multilinear mappings in topological vector spaces}, 
Studia Math. 39 (1971), 59--76. 

\bibitem[BoSi2]{bochnaksiciak} Bochnak, J., Siciak, J. \emph{Analytic functions in topological vector spaces}, Studia Math. 1971.

\bibitem[BFK]{friedankent2} Boucher, W., Friedan, D., Kent, A.,
\emph{Determinant formulae and unitarity for the $N=2$ superconformal 
algebras in two dimensions or exact results on string compactification},
Phys. Lett. B 172 (1986), no. 3-4, 316--322. 


\bibitem[Bo]{bourbakivariete}
Bourbaki, N., \emph{\'{E}l\'{e}ments de math\'{e}matique. Fasc. XXXIII. Vari\'{e}t\'{e}s diff\'{e}rentielles et analytiques. Fascicule de r\'{e}sultats (Paragraphes 1 \`{a} 7)}, 
Actualit\'{e}s Scientifiques et Industrielles, No. 1333 Hermann, Paris 1967, 97 pp.

\bibitem[CCTV]{varadarajan} Carmeli, C., Cassinelli, G., Toigo, A., Varadarajan, V. S., \emph{Unitary representations of super Lie groups and applications to the classification and multiplet structure of super particles}, 
Comm. Math. Phys. 263 (2006), no. 1, 217--258. 


\bibitem[DuSch]{dunford} Dunford, N., Schwartz, J. T.,
\emph{Linear operators. Part II. Spectral theory. Selfadjoint operators in Hilbert space}, With the assistance of William G. Bade and Robert G. Bartle. Reprint of the 1963 original. Wiley Classics Library.  John Wiley \& Sons, Inc., New York, 1988. pp. i--x, 859--1923 and 1--7. 

\bibitem[DeMo]{delignemorgan} Deligne, P., Morgan, J. W., \emph{Notes on supersymmetry (following Joseph Bernstein)}, Quantum fields and strings: a course for mathematicians, Vol. 1, 2 (Princeton, NJ, 1996/1997), 41--97, Amer. Math. Soc., Providence, RI, 1999.


\bibitem[FSZ]{FSZ81} Ferrara, S., Savoy, C.A., Zumino, B.,
\emph{General massive multiplets in extended supersymmetry.}
Phys. Lett. B 100(5) (1981), 393--398 .

\bibitem[FQS]{friedanqiu}  Friedan, D., Qiu, Z., Shenker, S.,
\emph{Superconformal invariance in two dimensions and the tricritical Ising model},
Phys. Lett. B 151 (1985), no. 1, 37--43. 

\bibitem[GKO]{goddard} Goddard, P., Kent, A., Olive, D., 
\emph{Unitary representations of the Virasoro and super-Virasoro algebras}, 
Comm. Math. Phys. 103 (1986), no. 1, 105--119.



\bibitem[HiPh]{hillephilips}
Hille, E., Phillips, R. S.
\emph{Functional analysis and semi-groups}, 
Third printing of the revised edition of 1957. American Mathematical Society Colloquium Publications, Vol. XXXI. American Mathematical Society, Providence, R. I., 1974. xii+808 pp. 

\bibitem[Io]{iohara} Iohara, K., \emph{Unitarizable highest weight modules of the $N=2$ super Virasoro algebras: untwisted sectors}, Lett. Math. Phys. 91 (2010), no. 3, 289--305.




\bibitem[JaKa]{jakobsen2} Jakobsen, H. P., Kac, V.,
\emph{A new class of unitarizable highest weight representations of infinite-dimensional Lie algebras. II},
J. Funct. Anal. 82 (1989), no. 1, 69--90. 

\bibitem[JaZh]{jarviszhang} Jarvis, P. D., Zhang, R. B.,
\emph{Unitary Sugawara constructions for affine superalgebras},
Phys. Lett. B 215 (1988), no. 4, 695--700. 


\bibitem[Jo]{jorgenson} Jorgensen, P. E. T., Moore R. T., \emph{Operator Commutation Relations}, Math. Appl., D. Reidel Publishing Co., Dordrecht,
1984.

\bibitem[KaTo]{kactodor} Kac, V. G., Todorov, I. T., 
\emph{Superconformal current algebras and their unitary representation}, Comm. Math. Phys. 102 (1985), 
337--347.


\bibitem[Ki]{kirillov} Kirillov, A. A., \emph{Representations of the infinite-dimensional unitary group}, 
Soviet Math. Dokl. 14 (1973), 1355--1358.

\bibitem[Ko]{kostant} Kostant, B., \emph{Graded manifolds, graded Lie theory, and prequantization}, Differential geometrical methods in mathematical physics (Proc. Sympos., Univ. Bonn, Bonn, 1975), pp. 177--306. Lecture Notes in Math., Vol. 570, Springer, Berlin, 1977.

\bibitem[Me]{merigon} Merigon, S., \emph{Integrating representations of Banach-Lie algebras},
J. Funct. Anal. 260 (2011), no. 5, 1463--1475.


\bibitem[Ne1]{khdifferentiable} Neeb, K.-H., \emph{On differentiable vectors for representations of infinite dimensional Lie groups}, J. Funct. Anal. 259 (2010), no. 11, 2814--2855.

\bibitem[Ne2]{neebanalytic} Neeb, K.-H., \emph{On analytic vectors for unitary representations of infinite dimensional Lie groups}, Ann. Inst. Fourier, to appear.

\bibitem[Ne3]{neebconflu} Neeb, K.-H., 
\emph{Semibounded representations and invariant cones in infinite dimensional Lie algebras}, 
Confluentes Math. 2 (2010), no. 1, 37--134.

\bibitem[Ne4]{Ne09} Neeb, K. -H., \emph{Semibounded unitary representations of 
infinite dimensional Lie
groups}, ``Infinite Dimensional Harmonic Analysis IV'',
Eds. J. Hilgert et al, World Scientific, 2009; 209--222 (to appear).

\bibitem[Nel]{nelson} Nelson, E.,
\emph{Analytic vectors}, Ann. of Math. (2) 70 1959 572--615. 



\bibitem[Ol]{olshanski} Ol'shanskii, G. I.,  \emph{Representations of infinite-dimensional classical groups, limits of enveloping algebras, and Yangians}, Topics in representation theory, 1--66, Adv. Soviet Math. 2, Amer. Math. Soc., Providence, RI, 1991.


\bibitem[ReSi]{reedsimon} Reed, M., Simon, B., \emph{Methods of modern mathematical physics. I. Functional analysis}, Academic Press, New York-London, 1972. xvii+325 pp.

\bibitem[Ru]{rudinfncth}
Rudin, W.,
\emph{Function theory in the unit ball of ${\C}^n$},
Grundlehren der Mathematischen Wissenschaften, 
241, Springer--Verlag, New York--Berlin, 1980. xiii+436 pp.

\bibitem[Ru86]{Ru86} Rudin, W., ``Real and Complex Analysis,'' 
McGraw-Hill, New York, 1986 



\bibitem[SaSt]{SS74}
Salam, A., Strathdee, J.,
\emph{Unitary representations of super-gauge symmetries.}
Nucl Phys. B 80 (1974), 499--505 .


\bibitem[Sal]{salmasian} Salmasian, H., \emph{Unitary representations of nilpotent super Lie groups},
Comm. Math. Phys. 297 (2010), no. 1, 189--227. 

\bibitem[Sau]{sauvageot} Sauvageot, F.
\emph{Repr\'{e}sentations unitaires des super-alg\`{e}bres de Ramond et de Neveu-Schwarz.},
Comm. Math. Phys. 121 (1989), no. 4, 639--657. 


\bibitem[Se]{segal}Segal, G., 
\emph{Unitary representations of some infinite-dimensional groups},
Comm. Math. Phys. 80 (1981), no. 3, 301--342. 
\bibitem[Si]{siciakLe}
Siciak J., \emph{A polynomial lemma and analytic mappings in topological vector spaces}, S\'{e}minaire Pierre Lelong (Analyse), (ann\'{e}e 1970--1971), pp. 131--142. Lecture Notes in Math., Vol. 275, Springer, Berlin, 1972.



\bibitem[Va]{varabook} Varadarajan, V. S.,
\emph{An introduction to harmonic analysis on semisimple Lie groups}, Corrected reprint of the 1989 original. Cambridge Studies in Advanced Mathematics, 16. Cambridge University Press, Cambridge, 1999. x+316 pp. 

\end{thebibliography}
\end{document}